\documentclass[10pt,twoside,a4paper,final]{amsart}

\usepackage[utf8]{inputenc}

\usepackage{amsthm}
\usepackage{amssymb}
\usepackage{amsfonts,a4}
\usepackage{amsxtra}
\usepackage{amscd}

\usepackage[centertags]{amsmath}
\usepackage{mathtools} 
\usepackage{mathabx} 
\usepackage{verbatim} 
\usepackage{color}
\usepackage[english]{babel}
\usepackage{algorithm}
\usepackage{algorithmic} %

\usepackage{tikz} 
\usepackage{tikz-cd}
\usetikzlibrary{calc}

\usepackage{nomencl} 
\makenomenclature 

\usepackage{ifthen} 
 \newboolean{extensions} 
 \setboolean{extensions}{false} 
 \newboolean{appendix} 
 \setboolean{appendix}{true} 

\usepackage{xargs}                      

\usepackage{todonotes}

\usepackage[mathscr]{eucal}
\usepackage{bbm}
\usepackage{enumitem}

\usepackage{graphicx}
\graphicspath{{./plots/}}

\usepackage{xspace} %
\usepackage[notcite,notref]{showkeys}

 \usepackage[scrtime]{prelim2e}

\newtheorem{thm}{Theorem}
\newtheorem{cor}[thm]{Corollary}
\newtheorem{lem}[thm]{Lemma}
\newtheorem{prop}[thm]{Proposition}

\newtheorem{df}[thm]{Definition}

\theoremstyle{definition}
\newtheorem{algo}[thm]{Algorithm}

\numberwithin{equation}{section}

\newcommand{\Aon}[1][n]{\ensuremath{\Ao_{n}}}
\newcommand{\abs}[1]{\ensuremath{\left|#1\right|}}
\newcommand{\asdefined}{\mathrel{=:}}

\newcommand{\bi}[1][\grid]{\mathfrak{B}_{#1}}
\newcommand{\bilin}[3][\grid]{\ensuremath{\bi[#1][#2,\,#3]}}

\newcommand{\Cleq}{\ensuremath{\lesssim}}

\newcommand{\CIPG}{$\textsf{C}^0\textsf{IPG}$\xspace}

\newcommand{\ACIPGM}{\textsf{A}$\textsf{C}^0\textsf{IPG}$\textsf{M}\xspace}

\newcommand{\definedas}{\mathrel{:=}}
\newcommand{\dual}[2]{\ensuremath{\left\langle #1,\,#2\right\rangle}}

\newcommand{\dx}{\ensuremath{\,\mathrm{d}x\xspace}}

\newcommand{\ds}{\ensuremath{\,\mathrm{d}s\xspace}}

\newcommand{\DG}[1][\grid]{\ensuremath{D_{\!\texttt{pw}}}}

\DeclareMathOperator{\divo}{div}
\newcommand{\DeltaG}[1][\grid]{\ensuremath{\Delta_{\!\texttt{pw}}}}
\newcommand{\DOFs}[1][\grid]{\ensuremath{\mathcal{N}_{#1}}}
\newcommand{\cDOFs}[1][\elm]{\ensuremath{\mathcal{N}}_{#1}^{\texttt{HTC}}}

\newcommand{\elm}{\ensuremath{K}\xspace}
\newcommand{\enorm}[2][\grid]{\left|\negthinspace\left|\negthinspace\left|{#2}%
                      \right|\negthinspace\right|\negthinspace\right|_{#1}}

\newcommand{\est}{\eta}
\newcommand{\ESTIMATE}{\textsf{ESTIMATE}}

\newcommand{\ovf}[0]{\Pi_0 f}

\newcommand{\grids}{\ensuremath{\mathbb{G}}\xspace}
\newcommand{\grid}{\mathcal{T}}

\newcommand{\gridk}[1][k]{\grid_{#1}}

\newcommand{\helm}[1][\ell]{\ensuremath{h_{\elm}}}
\newcommand{\hk}[1][k]{\ensuremath{h_{#1}}}

\newcommand{\hG}[1][\grid]{\ensuremath{h_{#1}}}

\newcommand{\ipol}{\mathcal{I}}
\newcommand{\ipolk}[1][k]{\ipol_{#1}}
\newcommand{\Cipolk}[1][k]{\mathcal{E}_{#1}}
\newcommand{\Gipolk}[1][k]{{\ipol}_{#1}}

\newcommand{\id}{\operatorname{id}}

\newcommand{\jump}[1]{\left[\negthinspace\left[{#1}\right]\negthinspace\right]}

\newcommand{\jumpn}[2]{\jump{{\partial^{#1}_n {#2}}}}
\newcommand{\meann}[2]{\mean{{\partial^{#1}_n {#2}}}}

\DeclareMathAlphabet{\lf}{OT1}{pzc}{m}{it}

\newcommand{\liftG}[1][k]{\mathcal{L}_{#1}}

\newcommand{\mean}[1]{\left\{ \kern -1.2mm \left\{ {#1}  \right\} \kern -1.2mm\right\}}
\newcommand{\marked}{\mathcal{M}}

\newcommand{\MARK}{\textsf{MARK}\xspace}

\newcommand{\nablaG}[1][\grid]{\ensuremath{\nabla_{\!\texttt{pw}}}}
\newcommand{\nablak}[1][k]{\ensuremath{\nabla_{k}}}
\newcommand{\N}{\ensuremath{\mathbb{N}}}

\newcommand{\neighk}[1][k]{\ensuremath{{N}_{#1}}}
\newcommand{\neighG}[1][\grid]{\ensuremath{{N}_{#1}}}

\newcommand{\nodes}[1][\grid]{\ensuremath{\mathcal{Z}_{#1}}}
\newcommand{\cnodes}[1][\grid]{\ensuremath{\mathcal{Z}_{#1}^{\texttt{HTC}}}}
\newcommand{\normal}{\ensuremath{\vec{n}}}

\newcommand{\norm}[2][\Omega]{\ensuremath{\left\|#2\right\|_{#1}}}
\newcommand{\snorm}[2][\Omega]{\ensuremath{\left|#2\right|_{#1}}}

\newcommand{\ovu}{ \overline u_\infty}
\DeclareMathOperator{\osc}{osc}

\providecommand{\Poincare}{{Poincar{\'e}}\xspace}
\renewcommand{\P}{\ensuremath{\mathbb{P}}}

\renewcommand{\paragraph}[1]{\noindent\raisebox{0pt}[10pt][0pt]{\textbf{#1.}}}
\newcommand{\pn}[1]{\partial_n {#1} }

\newcommand{\polspace}[1][2]{\ensuremath{\P}_{#1}} 
\newcommand{\cpolspace}[1][4]{\ensuremath{\hat\P}_{#1}} 

\newcommand{\REFINE}{\textsf{REFINE}\xspace}
\newcommand{\R}{\ensuremath{\mathbb{R}}}

\newcommand{\scp}[3][\Omega]{\ensuremath{\left\langle #2,\,#3\right\rangle}_{#1}}

\newcommand{\set}[1]{\left\{#1\right\}}
\newcommand{\side}{\ensuremath{F}\xspace}
\newcommand{\sides}{\begingroup \mathcal{F}\endgroup}
\newcommand{\sidesk}[1][k]{\begingroup \mathcal{F}_{#1}\endgroup}

\newcommand{\insides}[1][k]{{\begingroup\mathring{\mathcal{F}_{#1}}\endgroup}}

\newcommand{\skeletonk}[1][k]{\ensuremath{\Gamma_{#1}}}
\newcommand{\inskeletonk}[1][k]{\ensuremath{\mathring\Gamma_{#1}}}

\newcommand{\step}[1]{\noindent\raisebox{1.5pt}[10pt][0pt]{\tiny\framebox{$#1$}}\xspace}

\newcommand{\SOLVE}{\textsf{SOLVE}}
\DeclareMathOperator{\supp}{supp}

\renewcommand{\vec}[1]{\ensuremath{\boldsymbol{#1}}}

\newcommand{\V}{\ensuremath{\mathbb{V}}}

\newcommand{\VG}[1][\grid]{\V(#1)}

\newcommand{\wconv}{\rightharpoonup}

\begin{document}



\title[Convergence of an adaptive~\CIPG~method]{Convergence of an adaptive
  $C^0$-interior penalty \\ Galerkin method for the biharmonic problem}

\author[A.~Dominicus]{Alexander Dominicus}
\address{Corresponding Author: Alexander Dominicus,
 Fakult\"at f\"ur Mathematik,
 TU Dortmund University, 
 Vogelpothsweg 87, D-44227 Dortmund, Germany
 }%
 \email{alexander.dominicus@tu-dortmund.de}

\author[F.~Gaspoz]{Fernando Gaspoz}

\author[Ch.~Kreuzer]{Christian Kreuzer}


\keywords{Adaptive discontinuous Galerkin methods, quadratic
  C0-interior penalty method, convergence, biharmonic
  problem}

\subjclass[2010]{65N30, 65N12, 65N50, 65N15}

\maketitle
\begin{abstract}
We develop a basic convergence analysis for an adaptive 
\CIPG method
for the Biharmonic problem which provides convergence without
rates for all practically relevant marking strategies and all penalty
parameters assuring coercivity of the method. The analysis hinges on
 embedding properties of (broken) Sobolev and BV
 spaces, and the construction of a suitable  limit space.
 In contrast to the convergence result of adaptive
discontinuous Galerkin methods for elliptic PDEs, by Kreuzer and Georgoulis
(\cite{KreuzerGeorgoulis:17}), here we have to deal with the fact that
the Lagrange finite element spaces may possibly contain no proper
$C^1$-conforming subspace. This prevents from a straight forward
generalisation and requires the development of some new
key technical tools. 
\end{abstract}
\section{Introduction}\label{sec:introduction}

We develop here a basic convergence analysis for an adaptive $C^0$-interior
penalty method (\ACIPGM) for fourth order boundary value problems.
Let $\Omega \subset \R^2$ be a bounded polygonal domain with
Lipschitz boundary. For the ease of presentation we restrict
ourselves to the Biharmonic problem 
\begin{align}\label{eq:Biharmonic_strong}
    \Delta ^2 u =f \quad \text{in }\Omega, \quad\text{and}\quad
u=\frac{\partial u}{\partial \normal_\Omega}=0  \quad \text{on }\partial \Omega,
\end{align}
where $f \in L^2(\Omega)$ and
$\normal_\Omega$ denotes the outer normal on $\partial\Omega$. 
However, we emphasise that the presented techniques
also apply to more general
fourth order problems.  

Conforming discretisations of fourth order problems require 
$C^1$-elements
\cite{Argyris1968:The-TUBA-family,Ciarlet1974:Sur-lelement-de,DouglasDupontPercellScott:79}, 
which are typically very cumbersome to implement since they require
polynomial degree $\ge 5$ in $2d$ or constructions via
macrotriangulations. For 
this reason, mixed (see
e.g. \cite{Boffi2013:Mixed-finite-el,Ciarlet1974,
  Johnson1973:On-the-converge})
and non-conforming methods
(e.g. \cite{Bazeley1965:0.-C.-Zienkiewi,Morley1968:The-triangular-})
gained attraction. In this work, we consider the non-conforming so-called
$C^0$-interior penalty Galerkin discretisation (\CIPG)
of~\eqref{eq:Biharmonic_strong}. This method uses standard continuous Lagrange
finite elements of order $\ge 2$. Consistency is ensured and 
jumps of the normal derivatives across element interfaces are
penalised. For a thorough
introduction to $C^0$-interior penalty methods see
e.g. \cite{BrennerSung:05,EngelGarikipati:2002,HansboLarson:2002}.
A posteriori error estimators for the \CIPG~method were developed
in~\cite{GeorgoulisHoustonVirtanen:09,BrennerGudiSung2009} and
can be used to design an \ACIPGM based on the standard loop
\begin{align}\label{loop:SEMR}
  \SOLVE \to \ESTIMATE \to \MARK \to \REFINE.
\end{align}
The convergence theory, however, turns out to be a particular
challenging task for two reasons. First, the presence 
of the negative power of the mesh-size $h$ in the
discontinuity penalisation term. 
Second, the analysis of the \CIPG~method
suffers additionally from the fact that, in general, no conforming subspace with
proper approximation properties is available unless the polynomial
degree exceeds e.g. $4$ in $2d$; compare with
\cite{Boor1983:Approximation-b,GiraultScott:2002}.

The first issue also appears in adaptive discontinuous Galerkin
methods for $2$nd order problems. Here,
resorting to D\"orflers marking strategy, error
reduction~\cite{KarakashianPascal:07,HoppeKanschatWarburton:2008} 
and even optimal
convergence rates~\cite{BonitoNochetto:10}  of adaptive schemes are available.
These results generalise the ideas for conforming methods in
\cite{Doerfler:96,MoNoSi:00,Cascon2008:Quasi-optimal-c} 
based on the observation that the penalty is dominated by the
`conforming parts' of the estimator provided the penalisation
parameter is chosen sufficiently large. 
This idea was taken up in~\cite{Fraunholz:2015} in an
attempt to prove convergence of the \ACIPGM
for the biharmonic problem~\eqref{eq:Biharmonic_strong}, although the
resulting argument is unclear to hold: For example, it appears to us that,
in the proof of
the crucial estimator reduction property~\cite[Lemma
4.1]{Fraunholz:2015}, the possible increase of the penalty due to
refinement is not properly taken into account. 
However, there are 
generalisations of~\cite{BonitoNochetto:10} for the Hellan-Hermann-Johnson element
\cite{Huang2011:Convergence-of-} and a hybridisable $C^0$-discontinuous
Galerkin method~\cite{Sun2018:Quasi-optimal-c} where no negative power
of the mesh-size is present; compare also with the discussion
in~\cite{Cockburn2016:Contraction-pro}.

Very recently in \cite{KreuzerGeorgoulis:17} (c.f. also
\cite{KreuzerGeorgoulis:19})
the basic
convergence results for conforming adaptive finite element methods
\cite{MorinSiebertVeeser:08,Siebert:11} have been extended to adaptive
discontinuous Galerkin methods 
for $2$nd order problems. The result utilises a newly developed space limit of
the discrete space sequence created by the adaptive
loop~\eqref{loop:SEMR}.
Replacing Cea's Lemma in \cite{MorinSiebertVeeser:08}
by a version of the medius analysis of Gudi
\cite{Gudi:10} adapted to the limit space
yields convergence of discrete approximations 
to the weak solution in the limit space. Coincidence with the exact
solution follows thanks to properties of the marking strategy.
The result is neither restricted to
symmetric problems and discretisations nor to a particular marking
strategy and holds for all values of the penalty parameter, for which
the method is coercive.
This has important consequences in practical
computations: Since the condition number of the
respective stiffness matrix grows as the penalty parameter grows, the
magnitude of the penalisation affects the performance of iterative
linear solvers. This fact becomes even more relevant for the here
considered fourth order problem. 
We stress, however, that this technique does not provide linear
or even optimal convergence rates.

In this work, we extend~\cite{KreuzerGeorgoulis:17} to an
\ACIPGM for the Biharmonic problem~\eqref{eq:Biharmonic_strong}. The
main result states convergence of the adaptive loop~\eqref{loop:SEMR}
for most common marking strategies and  all penalty
parameters, for which the method is coercive.
Unfortunately,~\cite{KreuzerGeorgoulis:17}
makes exhaustive use of conforming subspaces
of the respective discrete spaces, which is
prohibitive for the \ACIPGM unless the polynomial degree of the
Ansatz space is large enough.
Therefore, the verification of certain properties of the limit space
requires the development of essentially different techniques and also
the convergence of 
discrete solutions cannot be concluded using the generalised medius
analysis of Gudi \cite{Gudi:10} from \cite{KreuzerGeorgoulis:17}.
For the sake of presentation, in this paper, we restrict ourselves to
quadratic $C^0$-elements.  
We emphasise,  however, that the techniques apply
to more general fourth order problems, arbitrary polynomial and even
discontinuous Galerkin discretisations, however, the construction of
suitable technical tools like interpolation operators and a posteriori
error estimators is getting much more involved.

The rest of this paper is organised as follows. In
Section~\ref{sec:adaptice_COIP_and_manin_result}, we introduce the \CIPG
discretisation and define the \ACIPGM  by a precise formulation of the adaptive
loop~\eqref{loop:SEMR}. We conclude the section stating the main result,
Theorem~\ref{thm:main}. For the sake of clarity, in
Section~\ref{sec:proof_main_result}, we first 
present the main ideas of its proof. The fact that the discrete
$C^0$-spaces do in general not contain proper $C^1$-conforming 
subspaces mainly affects the proofs of the two key technical results,
Lemma~\ref{lem:limit_space_is_Hilbert} and
Theorem~\ref{thm:u_k_to_u_infty}. They are presented in
Section~\ref{sec:proof_u_k_to_u_infty}.

\section{The adaptive \CIPG~finite element method and the main
  result}
\label{sec:adaptice_COIP_and_manin_result}
Let $\omega$ be a measurable set
 and $m\in\N$. We consider the usual Lebesgue
spaces $L^p(\omega;\R^m)$, $1\le p\le \infty$ over $\omega$ with
values in $\R^m$. In the case $p=2$, $L^2(\omega;\R^m)$ is a Hilbert
space  with inner product  $\scp[\omega]{\cdot}{\cdot}$ 
and associated norm $\norm[\omega]{\cdot}$. We also set $L^2(\omega):=L^2(\omega;\R)$.
The
Sobolev space $H^k(\omega)$ is the space of all functions in
$L^2(\omega)$ whose weak derivatives up to order $k$ are in $L^2(\omega)$. 
Thanks to the Poincar\'e-Friedrichs' inequality, the closure
$H_0^2(\omega)$ of $C_0^\infty(\omega)$ in $H^2(\omega)$ 
is a Hilbert space with inner product
$\scp[\omega]{D^2 \cdot}{D^2 \cdot}$ and norm
$\norm[\omega]{D^2 \cdot}$, where $D^2v$ denotes the Hessian of $v$.
The dual space $H^{-2}(\omega)$ of $H_0^2(\omega)$ is equipped with the norm
$\norm[H^{-2}(\omega)]{\lf{v}}:=\sup_{w\in
  H_0^2(\omega)}\frac{\dual{\lf{v}}{w}}{\norm[\omega]{D^2 w}}$, $\lf{v}\in H^{-2}(\omega)$, with
dual brackets defined by $\dual{\lf{v}}{w}:=\lf{v}(w)$, for
$w\in H^2_0(\omega)$.

For $f\in L^2(\Omega)$,
the weak formulation of~\eqref{eq:Biharmonic_strong} reads: find $u\in
H_0^2(\Omega)$, such that 
\begin{align}  \label{eq:Bihamonic_equation}
  a(u,v) = \int_\Omega fv \dx \quad \forall v \in
           H^2_0(\Omega),
\end{align}
for the bilinear form 
\begin{align*}
  a(w,v):=\int_\Omega D^2w \colon D^2v \dx = \int_\Omega \sum_{i,j=1}^2 \frac{\partial^2 w}{\partial 
  x_i \partial x_j} \frac{\partial^2 v}{\partial x_i \partial x_j}\dx ,
\end{align*}
which is uniformly coercive and continuous on $H_0^2(\Omega)$.
Consequently, Riesz' representation theorem provides  a unique
solution $u \in H^2_0(\Omega)$ of
\eqref{eq:Bihamonic_equation}. 

\subsection{The \CIPG~finite element Method}
 Let $\grid$ be a conforming and shape
regular subdivision of $\Omega$ into adjacent closed triangular elements
$\elm \in \grid$ such that $\overline \Omega = \bigcup\{\elm \colon\elm \in
  \grid\} $. 
Let
$\sidesk[\grid]:=\sides(\grid)$ be the set of one-dimensional faces
$\side$, associated with the subdivision $\grid$ (including $\partial
\Omega$), and let $\insides[\grid]$ be the
subset of interior sides only. The corresponding \textit{skeletons} are
then defined by $\skeletonk[\grid]
=\skeletonk[](\grid):= \bigcup\{\side \colon \side \in
  \sidesk[\grid]\} $ and $\inskeletonk[\grid]:=\bigcup\{\side \colon \side \in
  \insides[\grid]\} $  respectively. We assume that $\grid$ is derived by
iterative or recursive bisection of an initial
conforming mesh $\grid_0$; compare with \cite{Baensch:91,Kossaczky:94,Maubach:95}.
We denote by $\grids$ the family of shape-regular triangulations
consisting of such refinements of $\grid_0$. For
$\grid,\grid_\star\in\grids$, we write $\grid_\star\ge\grid$,
whenever $\grid_\star$ is a refinement of $\grid$.

For $r\ge 2$, we define the \textit{Lagrange finite-element
  space} by  
\begin{align*}
  \V(\grid)\definedas
  H^1_0(\Omega)\cap\mathbb{P}_r(\grid)\quad\text{with}\quad
  \mathbb{P}_r(\grid)\definedas \{v \in L^1(\Omega)\colon v|_\elm
  \in \mathbb{P}_r(\elm)~ \forall \elm \in \grid\}.
\end{align*}
Obviously, we have $\V(\grid)\subset H_0^1(\Omega)$ but
$\V(\grid)\not\subset H_0^2(\Omega)$ in general. Since each function
in $\V(\grid)$ is piecewise polynomial on 
$\grid$, we have, however, that 
\begin{align*}
  \V(\grid)\subset H_0^2(\grid)\definedas H^2(\grid)\cap H_0^1(\Omega),
\end{align*}
where $H^2(\grid):=\{v\in L^2(\Omega)\colon v|_\elm\in
H^2(\elm),~\forall\elm\in\grid\}$.

The  piecewise constant \textit{mesh-size} function $\hG:
\Omega \to \R_{\geq 0}$ is defined by $\hG(x):
=\hG[\elm]:=\abs{\elm}^{1/d}$ for $x \in \elm \setminus \partial \elm$ and
$\hG(x) :=\hG[\side]:=\abs{\side}^{1/(d-1)}$ for $x\in\side \in
\sides$.  Let $\nodes[\grid]$ be the set of Lagrange nodes of
$\V(\grid)$, which can be identified with its nodal degrees of freedom
$\DOFs$. For $z\in\overline{\Omega}$, we denote its 
neighbourhood by $\neighG(z):= \set{\elm ' \in \grid \mid z \in
  \elm'}$, and the corresponding domain is defined by 
$\omega_\grid(z):=\Omega(\neighG(z))
$. Hereafter we use
$\Omega(X)\definedas
  \bigcup\{\elm\mid\elm\in
  X\}
$ for a collection of elements $X$.
With a little abuse of notation, for an element $\elm \in \grid$ we define 
its \emph{$j$th neighbourhood} recursively by $\neighG^j(\elm)\definedas \set{\elm '
  \in \grid \mid \elm' \cap \neighG^{j-1} (\elm) \not =
  \emptyset}$, where we set $\neighG^0(\elm)\definedas\elm$, and the
corresponding domain by
$\omega_\grid^j(\elm)\definedas\Omega(\neighG^j(\elm))$. We shall skip
the superindex if $j=1$, e.g. we write 
 $\neighG(\elm)=\neighG^1(\elm)$ and
$\omega_\grid(\elm)=\omega_\grid^1(\elm)$ for simplicity.
For a side $\side \subset \sidesk[\grid]$, we set
$\omega_\grid (\side)\definedas\bigcup \set{\elm \in \grid  \mid \side
  \subset \elm}$. We extend the above definitions to subsets
$\mathcal{M}\subset\grid$ setting
\begin{align*}
  \neighG^j(\mathcal{M})\definedas \{\elm\in\grid\colon \exists
  \elm'\in \mathcal{M}~\text{such
  that}~\elm\in\neighG^j(\elm')\}.
\end{align*}

Note that the shape regularity and conformity
of $\grids$ implies local quasi-uniformity, i.e.
\begin{align*}
  \sup_{\grid \in \grids} \max_{\elm' \in N_\grid(\elm)}
  \frac{\abs{\elm}}{\abs{\elm'}} \Cleq 1\qquad \text{and   }\qquad \sup_{\grid \in \grids} \max_{\elm \in \grid} \#N_\grid(\elm) \Cleq
  1 .
\end{align*}
In the sequel we use the notation $ a \Cleq b$, when $a \leq Cb $ for a
constant $C>0$, which is independent of all essential quantities (e.g. the
mesh-size of $\grid)$. 

In order to formulate the discrete bilinear form, we first need to
introduce the so-called jumps and averages of vector- respectively
tensorfields on the skeleton $\skeletonk[\grid]$. In fact, for
$v\in \VG$, we define
\begin{align*}
  \jumpn{}{v}_\side\definedas \jump{\nabla v\cdot\normal} |_\side\definedas \nabla v|_{\elm_1}\cdot\normal_{\elm_1}+\nabla v|_{\elm_2}\cdot\normal_{\elm_2}
\end{align*}
for $\side\in\insides[\grid]$ and $\side=\elm_1\cap\elm_2$ with 
two adjacent elements  $\elm_1,\elm_2\in\grid$. If $\side
\subset\partial \elm\cap\partial\Omega$, then $\jumpn{}{v}_\side\definedas
\nabla v|_{\elm}\cdot \normal_\elm$. The average of the Hessian of
$v\in\VG$ is defined by
\begin{align*}
  \meann{2}{v}_\side\definedas\mean{(D^2v)\normal\cdot\normal}\definedas  \frac 1 2
  \left(
  D^2v|_{\elm_1}+D^2v|_{\elm_2}\right) \normal_{\elm_1}\cdot\normal_{\elm_1}
\end{align*}
whenever $\side\in\insides[\grid]$ with $\side=\elm_1\cap\elm_2$ and
$\meann{2}{v}_\side:=D^2v|_{\elm}\normal_\elm\cdot\normal_\elm$ for  sides $\side\subset
\partial\elm\cap\partial\Omega$. We stress that the above definitions
do not depend on the choice of the ordering of the elements $\elm_1$
and $\elm_2$. This is not true for 
\begin{align}\label{eq:jump2n}
  \jumpn{2}{v}_\side\definedas \jumpn{}{(\nabla v\cdot\normal_{\elm_1})}
  |_\side\quad\text{and}\quad \meann{}{v}_\side\definedas
   \frac 1 2\left(\nabla v|_{\elm_1}+\nabla
  v|_{\elm_2}\right) \cdot\normal_{\elm_1} 
\end{align}
for $\side\in\insides[\grid]$ with $\side=\elm_1\cap\elm_2$ for
adjacent $\elm_1,\elm_2\in\grid$. However, the two expressions will
only appear as products with each other, e.g. as $\jumpn{2}{v}_\side
\meann{}{w}_\side$ or as $\jumpn{2}{v}_\side^2$, which are then again unique.



For $v,w \in \V(\grid)$ we recall then the discrete bilinear form from
\cite{BrennerSung:05, BrennerGudiSung2009}
\begin{align*}
  \begin{aligned}
    \bilin[\grid]{v}{w}&:= \int_{\grid} D^2v \colon D^2 w \dx -
    \int_{\sidesk[\grid]} \meann{2}{v}\jumpn{}{w} +
   \meann{2}{w} \jumpn{}{v} \ds \\
    &+\int_{\sidesk[\grid]} \frac{\sigma}{\hG} \jumpn{}{v} \jumpn{}{w}
    \ds.
  \end{aligned}
\end{align*}
Here, we used the following abbreviations
\begin{align*}
  \int_\grid \cdot \dx := \sum_{\elm \in \grid} \int_\elm \cdot \dx
  \quad \text{and}\quad \int_{\sidesk[\grid] } \cdot \ds :=\sum_{\side
  \in \sidesk[\grid]} \int_\side  \cdot \ds,
\end{align*}
where on each element $\elm\in\grid$, the piecewise 
Hessian $(\DG^2v)|_\elm = D^2(v|_\elm)\in L^2(\elm)$ exists since $v \in
H^2_0(\grid)$, i.e. we have $\int_{\grid} D^2v \colon D^2
w \dx = \int_\Omega \DG^2 v\colon \DG^2 w\dx$.

For sufficiently large $\sigma$, we have from \cite{BrennerSung:05}
that $\bi$ is continuous 
and coercive on $\VG$ with respect to the \textit{energy norm} 
\begin{align*}
  \enorm[\grid]{v}^2:= \int_\grid D^2 v \colon D^2 v \dx +
  \int_{\sidesk[\grid]}\frac{\sigma}{\hG} \abs{\jumpn{}{v}}^2 \ds \quad \forall
  v \in H^2_0 (\grid).
\end{align*}
When we consider the norm on a subset $\mathcal{M}\subset\grid$, then
we simply replace $\grid$ by $\mathcal{M}$ in the above definition.
In the following, instead of $\int_\Omega \DG^2 v \colon
\DG^2 v \dx $, we will also write $\int_\Omega \abs{\DG^2 v}^2 \dx $
for brevity.

\begin{prop}[Continuity and coercivity]
\label{prop:coerc_and_cont_discrete}
Let $\grid \in \grids$, then there exists $\sigma_\star>0$, such that
for all $\sigma
>\sigma_\star$
there exist positive constants
$C_{\textsf cont}, C_{\textsf coer}$  such that 
\begin{align*}
  \bilin[\grid]{v}{w} \leq C_{\textsf{cont}}
  \enorm[\grid]{v}\enorm[\grid]{w}
\qquad\text{and}\qquad
    C_{\textsf{coer}} \enorm[\grid]{v}^2 \leq \bilin[\grid]{v}{v}. 
\end{align*}
for all $v,w\in\VG$. The constants $\sigma_\star$, 
$C_{\textsf cont}$, and $C_{\textsf coer}$ solely depend on the shape
regularity of $\grid$ and the polynomial degree $r$.
\end{prop}

 Since $\VG$ is a Banach space with the energy-norm
 $\enorm[\grid]{\cdot}$,  there exists a unique $u_\grid\in\VG$ with
\begin{align}\label{eq:C0IP_Problem}
  \bilin[\grid]{u_\grid}{v_\grid}=\int_\Omega fv_\grid \dx \qquad \forall v_\grid \in \V(\grid).
\end{align}
This is the  $C^0$\textit{-interior penalty Galerkin} approximation of 
\eqref{eq:Bihamonic_equation}, which depends continuously on $f$, i.e.
\begin{align}\label{est:uniform_stab_solution}
  \enorm[\grid]{u_\grid} \Cleq  \|f\|_{\Omega}
\end{align}
thanks to the following broken \Poincare-Friedrichs
inequalities; compare with \cite{Brenner:2003}.
\begin{prop}\label{prop:Friedrich_type_estimate_global}
Let $v \in \V(\grid)$, then we have
\begin{align*}
\snorm[H^1_0(\Omega)]{v}^2 \Cleq \sum_{\elm \in \grid}
  \snorm[H^2(\elm)]{v}^2 
+\sum_{\side \in \sides(\grid)} h_\side^{-1} \int_\side \jumpn{}{v}^2 \, ds \Cleq \enorm[\grid]{v}^2.
\end{align*}
\end{prop}

Unfortunately, $\bi$ cannot be applied to functions from
$H_0^2(\grid)$ since no trace of second derivatives is available.
For a side $\side \in
\sidesk[\grid]$  we therefore define a local lifting operator
$\liftG[\grid]^\side \colon L^1(\side) \to 
  \polspace[r-2] (\grid)
^{2 \times 2}$ by 
\begin{align}\label{df:local_liftings}
  \int_\Omega \liftG[\grid]^\side(\varphi) \colon \vec{\tau} \dx = \int_\side
 \mean{\vec{\tau}\normal\cdot\normal }  \varphi \ds  \quad \forall \vec{\tau} \in  
  \polspace[r-2] (\grid)^{2 \times 2},
\end{align}
where the support of $\liftG[\grid]^\side(\varphi)$ is given by
$\omega_\grid(\side)$. Using a trace estimate, we have that 
\begin{align}\label{est:local_liftings}
  \norm[{\Omega}]{\liftG[\grid]^\side(\varphi)} \Cleq 
   \norm[\side]{\hG^{-1/2} \varphi};
\end{align}
where the right-hand side is allowed to be infinity;
compare also with e.g.
\cite[Lemma 4.33]{DiPietroErn:12}. We define the
 \textit{global lifting operator}
$\liftG[\grid] \colon L^1(\skeletonk[\grid]) \to  
  \polspace[r-2] (\grid)^{2\times 2} $ by
\begin{align}\label{est:global_liftings}
  \liftG[\grid](\varphi) :=\sum_{\side \in \sidesk[\grid]}
  \liftG[\grid]^\side (\varphi)\quad\text{with}\quad \norm[\Omega]{\liftG[\grid](\phi)} \Cleq
  \norm[{\skeletonk[\grid]}]{\hG^{-1/2} \phi}.
\end{align}
Noting that $\pn{v}\in L^2(\skeletonk[\grid])$ for all $v\in
H^2_0(\grid)$, we can extend the bilinear form $\bi$ from $\VG$ to $H_0^2(\grid)$ by
\begin{align}\label{df:bilinear_form_VG}
  \begin{aligned}
  \bilin[\grid]{v}{w}&\definedas \int_{\grid} D^2v \colon D^2 w \dx -
  \int_\Omega\liftG[\grid](\jump{\pn{w}}) \colon \DG^2v
                         +\liftG[\grid](\jump{\pn{v}}) \colon \DG^2w \dx \\
&+\int_{\sidesk[\grid]} \frac{\sigma}{\hG} \jumpn{}{v} \jumpn{}{w} \ds.
  \end{aligned}
\end{align}
In what follows, $\bi$ refers always to this definition
  unless stated otherwise.

Discontinuous Galerkin spaces can be
embedded into the space of functions with bounded
variation; compare
e.g. with \cite[Lemma 2]{BuffaOrtner:09}.
In the context of \CIPG methods, 
this transfers to an embedding of the first derivatives; compare also
with \cite{LewNeff:2004}. In order to make the statement more precise,
we denote by 
$BV(\Omega)^2$ the Banach space of vector valued functions with bounded
variation equipped with the norm
\begin{align*}
  \norm[BV(\Omega)]{\cdot}=\norm[L^1(\Omega)]{\cdot}+|D\cdot|(\Omega). 
\end{align*}
For $v\in W^{1,2}(\Omega)$, we have that the total variation $D(\nabla v)$
is the measure representing the distributional derivative of $\nabla v$ with total
variation
\begin{align*}
    |D(\nabla v)|(\Omega)\definedas \sup_{\phi\in C_0^1(\Omega)^{2\times 2}, \norm[L^{\infty}(\Omega)\le
    1]{\phi}}\int_\Omega \nabla v \cdot \divo \phi\dx.
\end{align*}
Here $C_0^1(\Omega)^{2\times 2}$ denotes the space of continuously differentiable
functions with compact support in
$\Omega$.

\begin{prop} \label{prop:total_variation_of_gradient_bounded_by_energy}
  Let $v \in \V(\grid)$, then we have for the total variation of
  $\nabla v$ that 
\begin{align*}
|D (\nabla v)|(\Omega)\Cleq \int_\Omega \abs{\DG^2 v} \dx
  + \int_{\sides(\grid)} \abs{\jumpn{}{v}}\, ds \Cleq \enorm[\grid]{v}.
\end{align*}
\end{prop} 


\subsection{A posteriori error bounds}
From here on, we restrict ourselves to quadratic $C^0$-elements,
i.e., $r=2$ and introduce the a posteriori error estimators from~\cite{BrennerGudiSung2009}. 
For $v \in \V(\grid)$ and $\elm \in \grid$ let
\begin{align}\label{eq:element_estimator}
   \est (v, \elm):= \left( \int_\elm \hG^4 \abs{f}^2 \dx +
  \int_{\partial \elm \cap  \Omega } \hG \jumpn{2}{v}^2\ds +
\sigma^2 \int_{\partial \elm} \hG^{-1} \jumpn{}{v}^2\ds  \right)^{1/2}.
\end{align}
When $v=u_\grid$, we simply write $\est_\grid
(\elm):=\est(u_\grid,\elm)$. Moreover, for $\marked \subset \grid $, we
set
\begin{align*}
  \est_\grid(v,\marked):=\left( \sum_{\elm \in \marked}
  \est(v,\elm)^2\right)^{1/2}\quad\text{and}\quad
  \est_\grid(\marked)\definedas \est_\grid(u_\grid,\marked).
\end{align*}
From \cite[Theorem 3.1]{BrennerGudiSung2009}, we have
that~\eqref{eq:element_estimator} defines a reliable estimator.
\begin{prop}\label{prop:upper_bound_estimator}
  Let $u \in H^2_0(\Omega)$ be the solution
  of~\eqref{eq:Bihamonic_equation} and $u_\grid$ the discrete solution
  of~\eqref{eq:C0IP_Problem}. Then,
  \begin{align*}
    \enorm[\grid]{u-u_\grid}\Cleq \est_\grid(
    \grid),
  \end{align*}
where the constants in $\Cleq$  depend only on the shape regularity
of $\grid$.
\end{prop}
In \cite[Section 4]{BrennerGudiSung2009} $\est_\grid$ is also proved
to be efficient.
\begin{prop}\label{prop:lower_bounds_u}
  Let $u \in H^2_0(\Omega)$ be the solution of~\eqref{eq:Bihamonic_equation} and $\grid \in \grids$. Then, for all
  $v \in \V(\grid)$, we have 
   \begin{align*}
     \est_\grid(v,\grid)
     \Cleq  \enorm[\grid]{u-v}
        +\osc(\grid,f),
  \end{align*}
with data-oscillation defined by
\begin{align*}
  \osc(\grid,f)^2:= \sum_{{\elm \in
  \grid}}\osc(\elm,f)^2,\quad&\text{where} \quad  \osc(\elm,f)^2:=  \int_{\elm}\hG[\elm]^4\abs{f-\Pi_0 f}^2
  \dx. 
\end{align*}
Here,  $\ovf$ denotes the
$L^2(\Omega)$-orthogonal projection onto $\mathbb{P}_0(\grid)$,
\begin{align*}
\ovf|_\elm:=\frac{1}{\abs{\elm}} \int_\elm f \dx \quad \forall \elm \in \grid.
\end{align*}
\end{prop}

\subsection{The adaptive \CIPG method (\ACIPGM)} Now, we are in
the position to precisely formulate the adaptive algorithm~\eqref{loop:SEMR} based on the modules
$\SOLVE$, $\ESTIMATE$, $\MARK$ and $\REFINE$, which are described in
more detail below.
\begin{algo}[\ACIPGM]\label{alg:SEMR}
Let $\grid_0$ be an initial triangulation. The adaptive algorithm is
an iteration of the following form:
\begin{enumerate}[leftmargin=1cm]
\item \label{step:SOLVE} $u_k=\SOLVE(\V(\grid_k));$
\item $\{\est_{k}(\elm)\}_{\elm \in \grid_k}=\ESTIMATE(u_k,\grid_k);$
\item \label{it_marked} $\marked_k=\MARK \left( \{\est_{k}(\elm)\}_{\elm \in \grid_k}, \grid_k \right);$
\item $\grid_{k+1}=\REFINE (\grid_k, \marked_k);$ increment k and go to Step \ref{step:SOLVE}.
\end{enumerate}
\end{algo}
Here we have replaced the subscript
triangulations $\{\gridk\}_{k\in\N_0}$ with the
iteration counter $k$ in $\est_k(\gridk)=\est_{\gridk}(\gridk)$  for
brevity. Similar short hand notations will be frequently used below
when no confusion can occur, e.g. we write also
$\neighk^j(\elm)=\neighG[\gridk]^j(\elm)$.
Next, we comment on the modules $\SOLVE$, $\ESTIMATE$, $\MARK$ and $\REFINE$.

\textbf{\SOLVE.} For a given mesh $\grid$ we assume that
  \[
    u_\grid=\SOLVE(\V(\grid)) 
  \] is the exact \CIPG~ solution of problem~\eqref{eq:C0IP_Problem}.

\textbf{\ESTIMATE. }We suppose that
  \begin{align*}
    \set{\est_\grid(\elm)}_{\elm \in \grid}:=\ESTIMATE(u_\grid,\elm)
  \end{align*}
  is the elementwise error defined in~\eqref{eq:element_estimator}.

\textbf{\MARK. }We assume that the output
  \begin{align*}
    \marked:=\MARK(\set{\est_\grid(\elm)}_{\elm \in \grid},\grid)
  \end{align*}
  of marked elements satisfies
  \begin{align}
    \label{df:marked_elements}
    \est_\grid(\elm)\leq g(\est_\grid(\marked)),\qquad\text{for all
    $\elm \in \grid\setminus \marked$.} 
  \end{align}
  Here $g\colon \R^+ \to \R^+$ is a fixed function, which is
  continuous in 0, with $g(0)=0$.

  \textbf{\REFINE. }We assume for $\marked \subset \grid$ that 
  \begin{align*}
   \grid \le  \widetilde \grid:= \REFINE(\grid,\marked)\in \grids,
  \end{align*}
  such that
  \begin{align}
    \label{df:refined_elements}
    \elm \in \marked \quad \Rightarrow \quad \elm \in \grid \setminus \widetilde \grid,
  \end{align}
  i.e., each marked element is at least refined once.

\subsection{The main result}
The main result of this work states that the sequence of
\CIPG~finite element approximations produced by the~\ACIPGM (Algorithm
\ref{alg:SEMR}) converges to the
exact solution $u \in H^2_0(\Omega)$
of~\eqref{eq:Bihamonic_equation}.
From here on we will refer to $\enorm[\grid_k]{\cdot}$ as $\enorm[k]{\cdot}$.
\begin{thm}\label{thm:main}
  We have that 
  \begin{align*}
    \est_k(\grid_k) \to 0\quad \text{and }\quad
  \enorm[k]{u-u_k} \to 0 \quad \text{as }k \to \infty.
\end{align*}
\end{thm}

\section{Proof of the main result Theorem~\ref{thm:main}}
\label{sec:proof_main_result}
The proof of convergence of the \ACIPGM is
based on ideas of \cite{MorinSiebertVeeser:08,Siebert:11} for
conforming elements and its generalisation \cite{KreuzerGeorgoulis:17}
to adaptive discontinuous Galerkin methods for the Poisson problem.
For the sake of clarity, in this section, we present the main ideas of the proof of
Theorem~\ref{thm:main} following the ideas
of~\cite{KreuzerGeorgoulis:17}. In contrast to the latter result here
we are faced with the problem
that $\VG$ contains no proper conforming subspace. This requires
new techniques of proof for the two key auxiliary results,
Theorem~\ref{thm:u_k_to_u_infty} and
Lemma~\ref{lem:limit_space_is_Hilbert}, which proofs are postponed to
Section~\ref{sec:proof_u_k_to_u_infty} below.

\subsection{Sequence of Partitions}
Following~\cite{MorinSiebertVeeser:08,Siebert:11,KreuzerGeorgoulis:19},
we split the domain $\Omega$ into essentially two parts according to
whether the mesh-size function $\hG[k]\definedas\hG[\grid_k]$ vanishes or not.
In order to make this rigorous, we define the set of eventually
never refined elements by
\begin{align}\label{eq:leaf_nodes}
\grid^+\definedas \bigcup_{k \geq 0} \bigcap _{l \geq k}
  \grid_l\qquad\text{with corresponding domain}\qquad \Omega^+\definedas \Omega(\grid^+).
\end{align}
Additionally, we denote the complementary domain $\Omega^-=
\Omega \setminus \Omega^+$.

For $k \in \N_0,$ we define  $\gridk^{+}:=\grid_k \cap \grid^+ $ and
for $j\ge 1$ 
\begin{align*}
  \begin{aligned}
    \gridk^{j+}&\definedas \{\elm \in \grid_k\colon
    \neighk^j(\elm)\subset\gridk^+\}	
    = \{\elm \in \grid_k\colon
 \neighk(\elm)\subset\gridk^{(j-1)+}\}	
,
\\
\grid_k^{j -}&:=\grid_k \setminus  \grid_k^{j+},
\end{aligned}                       
\end{align*}
where we used $\gridk^{0+}\definedas\gridk^+$ and 
$\gridk^{0-}\definedas\gridk^-$ in the identities when
$j=0$. 
For the corresponding domains we denote
$\Omega_k^{j-}:=\Omega(\grid_k^{j-})$ and 
$\Omega_k^{j+}:=\Omega(\grid_k^{j+})$.
Moreover, we adopt the above notations  for the corresponding faces,
e.g. $\sides^{j-}:=\sides(\grid_k^{j-})$,
$\sides^{j+}:=\sides(\grid_k^{j+})$.  
We remark that we need the above definitions of $\gridk^{j-}$ and
$\gridk^{j+}$, $j>0$, 
for technical reasons. In fact, our analysis involves Cl{\'e}ment type
quasi-interpolations for which local stability estimates involve
neighbourhoods. However, for different but fixed 
$j$s the above sets behave asymptotically similar for $k\to
\infty$. To see this, the next key result from \cite[Lemma
4.1]{MorinSiebertVeeser:08} states that neighbours of never refined  elements are
eventually also never refined again.
\begin{lem}\label{lem:neighbourhood_G+}
  For $\elm \in \grid^+$ 
  there exists a constant $L=L(\elm) \in \N_0$
  such that   \begin{align*}
   \neighk(\elm)= \neighk[L](\elm)
  \end{align*}
for all $k \geq L$. In particular, we have $\neighk(\elm) \subset \grid^+$ for
all $k \geq L$.
\end{lem}

The next lemma essentially goes back to \cite[(4.15) and Corollary
4.1]{MorinSiebertVeeser:08}. 
\begin{lem}\label{lem:Omega_star_and_h_vanishing}
For $j \in \N_0$  we have $\lim_{k \to \infty}
  \norm[L^\infty(\Omega) ]{\hG[k] \chi_{\Omega^{j-}_k} }=0,$
where
  $\chi_{\Omega^{j-}_k}$  denotes the characteristic function of
  $\Omega^{j-}_k$.
  Moreover, $\abs{\Omega^{j -}_k \setminus
  \Omega^-} =\abs{\Omega^{+}\setminus
  \Omega^{j+}_k} \to 0$
as $k \to \infty$. 
\end{lem}
\begin{proof}
  In order to see that  $\abs{\Omega^{+}\setminus
  \Omega^{j+}_k} \to 0$
   as $k \to \infty$, we observe from Lemma~\ref{lem:neighbourhood_G+}
   that for $\ell\in\N$, there exists $L=L(\ell)\ge \ell$, such that
   $\grid_\ell^+\subset\grid_L^{j+}$ since $\grid_\ell^+$    
   contains only finite many elements. Consequently, we have
   \begin{align*}
     |\Omega^+\setminus\Omega_{L(\ell)}^{j+}|\le
     |\Omega^+\setminus\Omega_{\ell}^{+}|\to 0 \qquad\text{as}~\ell\to\infty,
   \end{align*}
   i.e. we have proved the claim for a subsequence. Since the sequence $\{|\Omega^{+}\setminus
     \Omega^{j+}_k|\}_k$ is monotone, it must vanish as a whole.

   The first claim follows for $j=1$ from \cite[Corollary
   3.3]{Siebert:11}. By shape regularity, we have for $j>1$ that
   \begin{align*}
     h_\elm\eqsim |\elm|^{1/2}\le
     |\Omega^{j-}_k\setminus\Omega^{1-}_k|^{1/2}\le |\Omega^{j-}_k\setminus\Omega^{-}|^{1/2}\qquad\text{for all}~\elm\in \gridk^{j-}\setminus\gridk^{1-}.
   \end{align*}
   Consequently, we have
   \begin{align*}
     \norm[L^\infty(\Omega) ]{\hG[k] \chi_{\Omega^{j-}_k} }&\le
     \norm[L^\infty(\Omega) ]{\hG[k] \chi_{\Omega^{1-}_k}
     }+\norm[L^\infty(\Omega) ]{\hG[k]
     \chi_{\Omega^{j-}_k\setminus\Omega_k^{1-}} }\\
     &\le
     \norm[L^\infty(\Omega) ]{\hG[k] \chi_{\Omega^{1-}_k}}+
     |\Omega^{j-}_k\setminus\Omega^{1-}_k|^{1/2}\to 0
   \end{align*}
   as $k\to\infty$, which concludes the proof.  
\end{proof}

\subsection{The limit space}
In this section we discuss the limit of the finite element spaces
$\V_k$. Following the ideas in
\cite[Section 3.2]{KreuzerGeorgoulis:19}, we define 
\begin{align*}
  \V_\infty:=\big\{v\in H^1_0(\Omega)& \mid \nabla v \in BV(\Omega)^2,~  v|_{\Omega^-}\in
  H^2_{\partial\Omega\cap\partial\Omega^-}(\Omega^-),~
  v|_\elm\in  \P_r(\elm),~\forall \elm\in\grid^+,\\
   &\quad\text{such that}~\exists \{v_k\}_{k\in\N_0}, v_k\in\V_k
    ~\text{with}~\lim_{k\to\infty}\enorm[k]{v-v_k}=0
    \\
  &\quad\text{and}~
    \limsup_{k\to\infty}\enorm[k]{v_k}<\infty
  \big\}.
\end{align*}
By $  H^2_{\partial\Omega\cap\partial\Omega^-}(\Omega^-)$ we denote
the space of functions from $H^2_0(\Omega)$ restricted to the 
domain $\Omega^-$. For a function $v \in\V_\infty$ the piecewise Hessian
$\DG^2 v \in L^2(\Omega)^{2 \times 2}$ is defined by
\begin{align}\label{df:D2pw}
  \DG^2 v|_{\Omega^-}:= D^2v|_{\Omega ^-} ~ \text{on}~ \Omega^-\quad \text{and}\quad
\DG^2 v|_\elm:= D^2v (x) |_\elm ~\text{on}~ \elm \in \grid^+,
\end{align}
and we have from $\enorm[k]{v-v_k}\to 0$ as $k\to\infty$, for the
distributional Hessian, that
\begin{align*}
  \langle D^2v, \vec\varphi \rangle=-\langle \nabla v, \divo
  \vec\varphi\rangle =\int_\Omega \DG^2 v
  \colon\vec\varphi\dx-\int_{\sides^+}\jumpn{}{v}\vec\varphi \normal\cdot\normal\ds.
\end{align*}
The fact that $\nabla\V_\infty\subset BV(\Omega)^2$ is
motivated by Proposition~\ref{prop:total_variation_of_gradient_bounded_by_energy}.

We will use the following bilinear form on $\V_\infty$: For $v,w \in \V_{\infty}$, we define
\begin{align*}
\langle v,w \rangle _\infty := \int_\Omega \DG^2 v \colon\DG^2 w \dx +
  \sigma \int_{\sides^+} h_+^{-1} \jumpn{}{ v} \jumpn{}{ w} \ds ,
\end{align*}
where we set $h_+:=h_{\grid^+}$ and $\sides ^+:= \sides(
\grid^+)$. 

The induced norm is denoted by $\enorm[\infty]{v}= \langle v,v \rangle
_{\infty}^{1/2}$. Note that from the definition of $\V_\infty$, we have
$\nabla v  \in BV(\Omega)^2$. 
Consequently, we have from \cite[Theorem 3.88]{Ambrosio:2000} that the
$L^1$-trace of $\nabla v$ exists for all 
$\side \in \sides$ and for all $k \in \N_0$. Therefore, the jump
terms are  measurable with respect to the $1$-dimensional
Hausdorff measure on $\sides$, and we are able to evaluate the $k$-norm
$\enorm[k]{v}$ for $v \in \V_\infty$.
\begin{prop}\label{prop:boundedness_of_energy_norm}
  For $v \in \V_\infty$, we have
  \begin{align*}
    \enorm[k]{v} \nearrow \enorm[\infty]{v}< \infty \quad \text{as }k \to \infty.
  \end{align*}
In particular, for fixed $\ell \in \N_0$,
let $\elm
\in \grid_\ell$; then, we have 
\begin{align*}
  \int_{\set{\side \in \sidesk \colon \side \subset \elm}}
  \hG[k]^{-1} \jumpn{}{v}^2 \ds \nearrow  \int_{\set{\side \in \sides^+ \colon \side \subset \elm}}
  \hG[+]^{-1} \jumpn{}{v}^2 \ds \quad \text{as }k \to \infty.
\end{align*}
\end{prop}

\begin{proof}
  The assertion follows along the same arguments used in 
  \cite[Proposition 12]{KreuzerGeorgoulis:17}, however, we provide the
  short 
  proof in order to keep the paper as self contained as possible.
  
  For $v \in \V_\infty$ there exists a sequence 
  $v_k \in \V_k$, ${k \in \N_0}$, such that
  $\enorm[k]{v-v_k}\to0$ as $k\to\infty$ and $\limsup_{k \to \infty}
  \enorm[k]{v_k}< \infty$. Therefore,
  $\{\enorm[k]{v}\}_{k\in\N_0}$ is bounded, since 
  $
  \enorm[k]{v} \leq
  \enorm[k]{v-v_k}+\enorm[k]{v_k}<\infty$ uniformly in $k$.
  For $m \geq k$ we have, by inclusion $\bigcup_{\side \in
    \sidesk} \side \subset \bigcup_{\side \in \sidesk[m]}\side$ and
  mesh-size reduction $\hG[k]^{-1} \leq \hG[m]^{-1}$, that
  \begin{align*}
  \int_{\sidesk} \hG[k]^{-1} \jumpn{}{v}^2 \ds \leq  \int_{\sidesk}
    \hG[m]^{-1} \jumpn{}{v}^2 \ds \leq  \int_{\sidesk[m]}
    \hG[m]^{-1} \jumpn{}{v}^2 \ds. 
  \end{align*}
  Consequently, we have $\enorm[k]{v} \leq \enorm[m]{v}$ and
  $\set{\enorm[k]{v}}_{k \in \N_0}$ converges. In
  particular, for $\epsilon>0$, the exists $L=L(\epsilon)\in \N_0$ such that
  for all $k\ge L$ and some sufficiently large $m> k$, we
  have  
  \begin{align*}
    \epsilon> \abs{\enorm[m]{v}^2-\enorm[k]{v}^2} &=\sigma \int_{\sidesk[m]
                                                    \setminus (\sidesk[k] \cap \sidesk[m])} \hG[m]^{-1}\jumpn{}{v}^2 \ds
                                                    -
                                                    \sigma \int_{\sidesk[k]
                                                    \setminus (\sidesk[k] \cap \sidesk[m])} \hG[k]^{-1}\jumpn{}{v}^2
                                                    \ds\\
                                                  &\geq {\sigma} \int_{\sidesk[k] \setminus \sidesk[k]^+}
                                                    \hG[k]^{-1}\jumpn{}{v}^2 \ds.
\end{align*}
This follows from the fact that $\hG[m]|_\side \leq 2^{-1}\hG[k]|_\side$ for all $\side \in \sidesk[m]
\setminus (\sidesk[k] \cap \sidesk[m])$, and  $\sidesk[k]^+ =
\sidesk[m] \cap \sidesk[k]$ for sufficiently large  $m>k$. 
Therefore,  $\int_{\sidesk[m] \setminus \sidesk[m]^+}
\hG[m]^{-1}\jumpn{}{v}^2 \ds \to 0$ as  $m \to \infty$ and thus
\begin{align*}
  \enorm[k]{v}^2&=\int_\Omega \abs{\DG^2 v}^2 \dx +\sigma \int_{\sidesk^+}
                  \hG[k]^{-1} \jumpn{}{v}^2 \ds + \sigma  \int_{\sidesk \setminus\sidesk^+}
                  \hG[k]^{-1} \jumpn{}{v}^2 \ds \\
                & \to \enorm[\infty]{v}+0\qquad\text{as  $k \to \infty$.}
\end{align*}

The second claim is a localised version and follows by analogous arguments.
\end{proof}

The next lemma is crucial for
the existence of a generalised Galerkin solution in $\V_\infty$, its
proof is postponed to Section \ref{sec:limit_space_is_Hilbert}.
\begin{lem}\label{lem:limit_space_is_Hilbert}
  The space $(\V_\infty, \scp[\infty]{\cdot}{\cdot})$ is a Hilbert space.
\end{lem}
In order to extend the discrete problem~\eqref{eq:C0IP_Problem} to the
space $\V_\infty$, we have to extend the bilinear form $\bi$ to the space
$\V_\infty$. To this end, we define suitable liftings for the limit
space. Thanks to Lemma~\ref{lem:neighbourhood_G+}, for each $\side \in
\sides^+$, there exists $L=L(\side)$ such 
that $\side \in \sidesk[\ell]^{1+}$ for all $\ell\ge L$. We define the local lifting
operators 
\begin{align}
  \label{eq:df_local_lifting_infty}
  \liftG[\infty]^\side\definedas\liftG[L]^\side= \liftG[\grid_L]^\side.
\end{align}
From the definition of the discrete local liftings
\eqref{df:local_liftings}, we see that $\liftG[\infty]^\side$ vanishes
outside the two neighbouring element $\elm',\elm$, with $\side =\elm
\cap \elm'$. Consequently, we have $\liftG[\ell]^\side= \liftG[L]^\side
$ for all $\ell \geq L$, and therefore this definition is unique. The global lifting operator is
defined by
\begin{align}
  \label{eq:df_global_liftings_infty}
  \liftG[\infty]=\sum_{\side \in \sides^+} \liftG[\infty]^\side.
\end{align}
From estimate~\eqref{est:global_liftings} we have that $\sum_{\side \in
  \sides^+}\liftG[\infty]^\side (\jumpn{}{ v})$ is a
Cauchy sequence in $L^2(\Omega)^{2 \times 2}$. Therefore,
$\liftG[\infty](\jumpn{}{ v})\in L^2(\Omega)^{2 \times 2}$  and the estimate
\begin{align}\label{est:global_lifting_infty}
  \norm[\Omega]{\liftG[\infty](\jump{ \pn{v}})}\Cleq
  \norm[\Gamma^+]{h_+^{-1/2} \jumpn{}{ v}}
\end{align}
holds. Here we used the notation
$\Gamma^+:=\bigcup \set{\side \mid \side \in \sides^+}$. Now we are
in  position to generalise the DG-bilinear form to $\V_\infty$ setting
\begin{align*}
  \begin{split}
  \bilin[\infty]{v}{w}:&= \int_{\Omega} \DG^2v \colon \DG^2 w \dx -
  \int_\Omega\liftG[\infty](\jump{\pn{w}}) \colon \DG^2v
                         +\liftG[\infty](\jump{\pn{v}}) \colon \DG^2w \dx \\
&+\int_{\sides^+} \frac{\sigma}{\hG} \jumpn{}{v} \jumpn{}{w} \ds,
  \end{split}
\end{align*}
for all $v,w \in \V_\infty$.
\begin{cor}\label{cor:unique_solution_limit}
  There exists a unique $u_\infty\in \V_\infty$, such that 
  \begin{align}
    \label{eq:C0IP_limit}
    \bilin[\infty]{u_\infty}{v}=\int_\Omega fv \dx \quad \forall v \in \V_\infty.
  \end{align}
\end{cor}
\begin{proof}
 From Lemma~\ref{lem:limit_space_is_Hilbert} we have that $\V_\infty$ is a Hilbert space. Moreover, stability of the lifting
 operators~\eqref{est:global_lifting_infty}  and local scaled trace
 inequalities prove coercivity and continuity
 of $\bilin[\infty]{\cdot}{\cdot}$
 with respect to
 $\enorm[\infty]{\cdot}$; compare also with
 Proposition~\ref{prop:coerc_and_cont_discrete}. The assertion follows
 from the Riesz 
 representation theorem. 
\end{proof}
The following Theorem states that the solution
of~\eqref{eq:C0IP_limit} is indeed the limit of the adaptive sequence
produced by the \ACIPGM. Its proof is postponed to Section \ref{sec:proof_u_k_to_u_infty}.
\begin{thm}
\label{thm:u_k_to_u_infty}
  Let $u_\infty$ the solution of~\eqref{eq:C0IP_limit} and let
  $\{u_k\}_{k\in\N_0}$ be the sequence of \CIPG solutions produced by \ACIPGM. Then,
  \begin{align*}
    \enorm[k]{u_\infty -u_k} \to 0 \quad \text{as }k \to \infty.
  \end{align*}
\end{thm}

\subsection{Proof of Theorem~\ref{thm:main}}
In this section the marking strategy~\eqref{df:marked_elements}
becomes important. In particular, it essentially forces the maximal indicator to
vanish, which allows to control the error on the sequence
$\{\gridk^+\}_{k \in \N_0}$. Moreover, this has implications on the regularity
of the Galerkin solution $u_\infty\in\V_\infty$ from
Corollary~\ref{cor:unique_solution_limit}, which finally allow us to
prove that $u=u_\infty$. Thanks to the lower bound, we can thus
conclude the proof of  Theorem~\ref{thm:main} from
Theorem~\ref{thm:u_k_to_u_infty} employing the lower 
bound in Proposition~\ref{prop:lower_bounds_u}.

We start with proving that the maximal indicator vanishes.

\begin{lem}\label{lem:etamax->0}
  We have that
  \begin{align*}
    \max_{\elm\in\gridk}\eta_{k}(u_k,\elm)\to 0\quad\text{as}~k\to\infty
  \end{align*}
\end{lem}
\begin{proof}
  Let $k\in \N_0$, and $\elm\in\gridk^-$ such that
  $\eta_{k}(u_k,\elm)=\max_{\elm'\in\gridk^-}\eta_{k}(u_k,\elm')$. 
  Then we have by standard
  scaled trace- and inverse estimates that 
  \begin{align*}
    \eta_\grid(u_k,\elm)^2&= \int_\elm \hG^4 \abs{f}^2 \dx +
                            \int_{\partial \elm \cap  \Omega } \hk \jumpn{2}{u_k}^2 +
                            \sigma^2 \int_{\partial \elm} h_k^{-1} \jumpn{}{u_k}^2
    \\
                          &\Cleq \int_\elm \hG^4 \abs{f}^2 \dx + \int_{\omega_k(\elm)}|\DG^2
                            u_k|^2\dx +\sigma^2\int_{\partial \elm}
                            h_k^{-1} \jumpn{}{u_k}^2.
  \end{align*}
  The first term on the right hand side converges to zero thanks to
  Lemma~\ref{lem:Omega_star_and_h_vanishing}. For the remaining terms,
  we have from triangle inequalities that 
  \begin{multline*}
    \int_{\omega_k(\elm)}|\DG^2
                            u_k|^2\dx +\sigma^2\int_{\partial \elm}
                            h_k^{-1} \jumpn{}{u_k}^2
                            \\
                            \begin{aligned}
                              &\Cleq 
                            \enorm[k]{u_\infty-u_k}^2
                            +\int_{\omega_k(\elm)}|\DG^2
                            u_\infty|^2\dx+\sigma^2\int_{\partial \elm} h_k^{-1}
                            \jumpn{}{u_\infty}^2
    \\
                          &\le 
                            \enorm[k]{u_\infty-u_k}^2+\int_{\omega_k(\elm)}|\DG^2
                            u_\infty|^2\dx+\sigma^2\int_{\sides^+\setminus\sides_k^+} h_+^{-1}
                            \jumpn{}{u_\infty}^2
                            \end{aligned}
  \end{multline*}
  and also these terms vanish thanks
  to~Theorem~\ref{thm:u_k_to_u_infty} and
  Lemma~\ref{lem:Omega_star_and_h_vanishing} thanks to uniform
  integrability of $\DG^2 u_\infty$. Consequently
  \begin{align*}
    \max_{\elm\in\gridk^-}\eta_{k}(u_k,\elm)\to 0\qquad\text{as}~k\to\infty.
  \end{align*}
  Thanks to the marking strategy~\eqref{df:marked_elements},  this
  proves the assertion.
  \end{proof}

\begin{lem}\label{lem:est_Gk+}
  We have
 $    \est_k(\gridk^{1+}) \to 0$ as $k \to \infty$.
\end{lem}

\begin{proof}
Employing Lemma~\ref{lem:etamax->0}, the claim follows 
from reformulating the estimator in an 
integral framework and a generalised Lebesgue dominated
convergence theorem; for details see \cite[Proposition 4.3]{MorinSiebertVeeser:08}.
\end{proof}

Lemma~\ref{lem:est_Gk+} yields in particular that
\begin{align}
  \label{eq:jump+->0}
  \int_{\sides_k^+} h_k^{-1} \jumpn{}{u_k}^2\to 0\quad\text{as}~k\to\infty.
\end{align}
We shall use this fact to conclude additional regularity of the limit function.

\begin{lem}
  We have for $u_\infty\in\V_\infty$ from
  Corollary~\ref{cor:unique_solution_limit} that
  $u_\infty\in H_0^2(\Omega)$.
\end{lem}

\begin{proof}
  From Theorem~\ref{thm:u_k_to_u_infty}, we know that $u_k\to
  u_\infty$ in $H_0^1(\Omega)$ as well as 
  \begin{align*}
    \DG^2 u_{k}\to \DG^2 u_\infty\qquad\text{in}~L^2(\Omega)^{2\times
    2}\quad\text{as}~k\to\infty. 
  \end{align*}
  We have  that
  the distributional Hessian of $u_{k}$ is given by 
  \begin{align*}
    \langle D^2 u_k, \vec\varphi \rangle&= \int_{\Omega} \DG^2 u_k
                                          \colon \vec\varphi \dx -\int_{\sides_k} \vec\varphi \normal \cdot \normal \jump{\pn
      u_k} \ds,\quad \vec\varphi\in C_0^\infty(\Omega)^{2\times 2}.
  \end{align*}
  Consequently, $u_\infty$ has second weak derivatives $\DG^2 u_\infty
  $ if and only if the latter
  term vanishes as $k\to\infty$. This follows from 
  \begin{align}\label{eq:jumps->0}
    \int_{\sides_k} h_k^{-1}\jump{\pn
      u_k}^2 \ds\to 0\quad\text{as}~k\to\infty,
  \end{align}
  which even implies $u_\infty\in H_0^2(\Omega)$ since $\sides_k$
  contains also boundary sides. In order to
  verify~\eqref{eq:jumps->0}, we estimate
  \begin{align*}
    \int_{\sides_k} h_k^{-1}\jump{\pn u_k}^2 \ds&= \int_{\sides_k^-} h_k^{-1}\jump{\pn
      u_k}^2 \ds+\int_{\sides_k^+} h_k^{-1}\jump{\pn
                                                  u_k}^2 \ds
    \\
                                                &\leq 2 \int_{\sides_k^-} h_k^{-1}\jump{\pn
                                                  u_\infty}^2 \ds
                                                  +2\enorm[k]{u_\infty-u_k}^2+
                                                  \int_{\sides_k^+} h_k^{-1}\jump{\pn
      u_k}^2 \ds 
  \end{align*}
  Thanks to Proposition~\ref{prop:boundedness_of_energy_norm},
  Theorem~\ref{thm:u_k_to_u_infty} and~\eqref{eq:jump+->0}, we have
  that all three terms tend to zero. This proves the assertion.
\end{proof}

\begin{lem}\label{lem:u=uinfty}
  Let $u\in H_0^2(\Omega)$ and $u_\infty\in \V_\infty$ the solutions
  of~\eqref{eq:Bihamonic_equation} and~\eqref{eq:C0IP_limit}
  respectively. Then $u=u_\infty$.
\end{lem}
\begin{proof}
  We recall that for $v,w\in H_0^2(\Omega)$ we have
  $\bilin[]{v}{w}=\bilin[k]{v}{w}=\bilin[\infty]{v}{w}$. Therefore, we
  obtain from $u_\infty\in H_0^2(\Omega)$ and \eqref{eq:C0IP_limit}
  that
  \begin{align*}
    \enorm[]{u-u_\infty}^2&\Cleq \bilin[]{u-u_\infty}{u-u_\infty}\\
                          &= \bilin[]{u}{u-u_\infty}-\bilin[\infty]{u_\infty}{u}+\bilin[\infty]{u_\infty}{u_\infty}
    \\
                          &
                            =\langle f,u-u_\infty\rangle - \bilin[\infty]{u_\infty}{u}+\langle
                            f, u_\infty\rangle
    \\
                          &=\langle f,u\rangle
                            -\bilin[k]{u_\infty}{u}=\langle f,u\rangle
                            -\bilin[k]{u_k}{u}+\bilin[k]{u_\infty-u_k}{u}
    \\
                          &\le\langle f,u\rangle
                            -\bilin[k]{u_k}{u}+\enorm[]{u}\enorm[k]{u_\infty-u_k}.
  \end{align*}
  The last product vanishes thanks to
  Theorem~\ref{thm:u_k_to_u_infty} and we are left with the remaining
  parts.  By the density of $H_0^3(\Omega)$ in $H_0^2(\Omega)$ for
  $\epsilon>0$, we may choose $u_\epsilon\in H_0^3(\Omega)$ such that
  $\enorm[]{u-u_\epsilon}\le \epsilon$. Recalling that $\langle f,v_k\rangle
  -\bilin[k]{u_k}{v_k}=0$ for all
  $v_k\in\V_k$, we employ  
  standard a posteriori techniques
  \cite{BrennerGudiSung2009,GeorgoulisHoustonVirtanen:09,Verfurth:2013}
  to obtain
  \begin{multline*}
    \abs{ \langle f,u\rangle
      -\bilin[k]{u_k}{u}}\\
    \begin{aligned}
      &\leq \abs{ \langle f,u_\epsilon\rangle
      -\bilin[k]{u_k}{u_\epsilon}}+\abs{ \langle f,u-u_\epsilon \rangle
      -\bilin[k]{u_k}{u-u_\epsilon}}   \\
      &\Cleq 
      \sum_{\elm\in\gridk^-}\eta_k(u_k,\elm)\|\hk\|_{L^\infty(\Omega^{1-})}|u_\epsilon|_{H^3(\omega^2_k(\elm))}
      \\
      &\quad
      +\sum_{\elm\in\gridk^+}\eta_k(u_k,\elm)\enorm[{\neighk^2(\elm)}]{u_\epsilon}+\epsilon
      \|f\|_{L^2(\Omega)}
      \\
      &\Cleq \|\hk\|_{L^\infty(\Omega^{1-})}\eta_k(u_k,\gridk^-)
      |u_\epsilon|_{H^3(\omega^2_k(\elm))}+\eta_k(u_k,\gridk^+)
      \enorm[]{u_\epsilon}+\epsilon \|f\|_{L^2(\Omega)}.
    \end{aligned}
  \end{multline*}
  Here, we have used interpolation estimates in $H^3$ for the first
  term and stability of the interpolation for the second term as well
  as~\eqref{est:uniform_stab_solution} and the finite overlap of the
  neighbourhoods.
 The first term on the right
  hand side vanishes thanks to
  Lemma~\ref{lem:Omega_star_and_h_vanishing} and since the estimator
  stays bounded (Proposition~\ref{prop:lower_bounds_u}). The second
  term vanishes thanks to Lemma~\ref{lem:est_Gk+}. Combining the above
  findings, we obtain by letting $k\to\infty$ that
  \begin{align*}
    \enorm[]{u-u_\infty}^2\Cleq \epsilon\|f\|_{L^2(\Omega)}.
  \end{align*}
  Since $\epsilon$ was arbitrary, this proves the assertion.
\end{proof}

For the second term, we may use stability

\begin{proof}[Proof of Theorem~\ref{thm:main}]
  Thanks to Lemma~\ref{lem:u=uinfty} and Theorem~\ref{thm:u_k_to_u_infty},
    we have that $\enorm[k]{u-u_k}\to 0$ as $k\to\infty$. 
  
  Combining the lower bound  Proposition~\ref{prop:lower_bounds_u}
  with Lemmas~\ref{lem:est_Gk+},~\ref{lem:u=uinfty} and~\ref{lem:Omega_star_and_h_vanishing}, we obtain
  \begin{align*}
    \est_k(\gridk)^2&\Cleq  \enorm[k]{u-u_k}^2
                      +\osc(\gridk,f)^2\\
    &= \enorm[k]{u-u_k}^2+
    \int_{\gridk ^-}\hG[k]^4\abs{f-\Pi_0
    f}^2+\int_{\gridk^+}\hG[k]^4\abs{f-\Pi_0 f}^2
    \\
    &\le
      \enorm[k]{u-u_k}^2+\norm[L^\infty(\Omega)]{\hG[k]\chi_{\Omega_k^-}}^4\norm{f}^2+\eta_k(u_k,\gridk^+)^2
    \\
    &\to 0
  \end{align*}
  as $k\to \infty$. Here we have used the boundedness
  $\int_{\elm}\hG[k]^4\abs{f-\Pi_0 f}^2\dx\le \eta_k(u_k,\elm)^2$.
\end{proof}

\section{Proofs of Lemma~\ref{lem:limit_space_is_Hilbert} and Theorem \ref{thm:u_k_to_u_infty}}
\label{sec:proof_u_k_to_u_infty}

In order to close the proof of the main result,
Theorem~\ref{thm:main}, we need to verify
Lemma~\ref{lem:limit_space_is_Hilbert} and Theorem
\ref{thm:u_k_to_u_infty}. The former states that $\V_\infty$ is a
Hilbert space with norm $\enorm[\infty]{\cdot}$, and thus a unique
solution $u_\infty\in\V_\infty$ of~\eqref{eq:C0IP_limit} exists;
see Corollary~\ref{cor:unique_solution_limit}. 
The latter proves that $u_\infty$ is indeed the limit of the \CIPG
approximations $\{u_k\}_{k\in\N_0}$ produced by the \ACIPGM.
We emphasise that in contrast to~\cite{KreuzerGeorgoulis:17}, the lack
of proper $H^2$-conforming subspaces of \CIPG spaces, does not allow
for a straight forward generalisation: For example, in order to
prove $\enorm[k]{u_\infty-u_k}\to 0$, 
in \cite{KreuzerGeorgoulis:17} the best-approximation property for inf-sup
stable conforming elements \cite{MorinSiebertVeeser:08,Siebert:11} is replaced by
a variant of Gudi's medius analysis \cite{Gudi:10}.
However, this required a discrete smoothing operator into
$\V_\infty$, whose construction is heavily based on the existence of a
proper conforming subspace of $\V_k$.

After recalling auxiliary \Poincare -
and Friedrichs-type inequalities, we shall introduce a
smoothing operator, which maps $\V_k$ into
$H_0^2(\Omega)$. This accounts for the fact that each 
$v\in\V_\infty$ on $\Omega^-$ is a restriction of an $H_0^2(\Omega)$
function. Moreover, we require an interpolation operator in order to 
deal with the piecewise discrete structure of $\V_\infty$ on
$\Omega^+$. Both operators need to satisfy some compatibility
conditions. 
Finally, we conclude the section with the proofs of
Lemma~\ref{lem:limit_space_is_Hilbert} and  Theorem
\ref{thm:u_k_to_u_infty}.

\subsection{Preliminary results}
\label{sec:auxiliary}

The following \Poincare and Friedrichs estimates are
subsequently used to prove stability  of the smoothing and
quasi-interpolation operators, defined below.  
\begin{lem}\label{lem:Poincare_estimate_local}
Let $\grid,\grid_*$ be some triangulations of $\Omega$ with
$\grid\le\grid_*$ and let $v \in \V(\grid_*)$. Moreover, for
$\elm\in \grid$ let
$D_\elm \subset \Omega$ be either $\omega_\grid(\elm)$ or $\omega_\grid^2(\elm)$.
 Then, there exists a linear polynomial $Q$, defined on $D_\elm$  such that we have 
 \begin{subequations}
   \begin{align}\label{eq:Poincare}
     \snorm[H^1(D_\elm)]{v-Q}^2 \Cleq \int_{D_\elm}
     \hG^2   \abs{\DG^2 v}^2  \dx+  
     \sum_{\substack{ \side \in \sides( \grid_*) \\  \side \subset
     D_\elm}} 
     \int_{  \side} \hG^2  \hG[\grid_*]^{-1}  \jumpn{}{v}^2 \ds.
   \end{align}
   If additionally $ \side \subset {D_\elm} \cap \partial
   \Omega$ for some $\side \in \sidesk[\grid]$, then
   \begin{align}
     \label{eq:Friedrichs}
     \snorm[H^1(D_\elm)]{v}^2 \Cleq \int_{D_\elm}
  \hG^2   \abs{\DG^2 v}^2  \dx+  
  \sum_{\substack{ \side \in \sides( \grid_*) \\  \side \subset
 D_\elm
     }} 
\int_{  \side} \hG^2  \hG[\grid_*]^{-1}  \jumpn{}{v}^2 \ds.
  \end{align}
\end{subequations}
\end{lem}

\begin{proof}
Let $Q \in \polspace[1](D_\elm)$ be the
$H^1$-orthogonal projection of $v$ into $\polspace[1](D_\elm)$,
i.e.,
\begin{align*}
  \scp[D_\elm]{\nabla(v-Q)}{\nabla P}=0 \quad \forall P \in
  \polspace[1](D_\elm)\qquad \text{and}\qquad \int_{D_\elm} Q\dx =\int_{D_\elm} v \dx.
\end{align*}
Now the proof of~\eqref{eq:Poincare} is a direct consequence of
\cite[Proposition 1]{KreuzerGeorgoulis:17}.

The second claim~\eqref{eq:Friedrichs} follows from \cite[Corollary 4.3]{BuffaOrtner:09} together with
\cite[Proposition 1]{KreuzerGeorgoulis:17} and the definition of the jump
terms on boundary sides.
\end{proof}

The following lemma extends the previous result to the limit space $\V_\infty$.
\begin{lem}[\Poincare-Friedrichs $\V_\infty$]
\label{lem:poincare_local_limit}
  Let   $v \in
  \V_\infty$ and let either $D_\elm =\omega_k(\elm)$ or
  $D_\elm=\omega_k^2(\elm)$ for some $\elm\in \gridk$ and $k\in\N_0$.
  Then, there exists $Q\in\polspace[1](D_\elm)$, such that
    \begin{align*}
      \snorm[H^1(D_\elm)]{v-Q}^2 \Cleq \int_{D_\elm}
      \hG[k]^2 \abs{\DG^2 v}^2 \dx + \sum_{\substack{\side \in \sides^+
      \\
      \side \subset D_\elm}}
      \int_\side \hG[k]^2 \hG[+]^{-1} \jumpn{}{v}^2 \ds.
    \end{align*}
    If in addition 
    $ \side \subset {D_\elm} \cap \partial \Omega$ for some
    $\side \in \sidesk[\grid]$, then
    \begin{align*}
      \snorm[H^1(D_\elm)]{v}^2 \Cleq \int_{D_\elm}
      \hk^2   \abs{\DG^2 v}^2  \dx+  
      \sum_{\substack{ \side \in \sides^+ \\  \side \subset
      D_\elm}} 
      \int_{  \side} \hk^2  \hG[+]^{-1}  \jumpn{}{v}^2 \ds.
    \end{align*}

\end{lem}
\begin{proof}
  We follow the ideas of \cite[Lemma 13]{KreuzerGeorgoulis:17} and
  let $Q \in \polspace[1](D_\elm)$ be the
$H^1$-orthogonal projection of $v$ into $\P_1(D_\elm)$,
defined by
\begin{align*}
  \scp[D_\elm]{\nabla(v-Q)}{\nabla P}=0 \quad \forall P \in
  \polspace[1](D_\elm)\qquad \text{and}\qquad \int_{D_\elm} Q\dx =\int_{D_\elm} v \dx.
\end{align*}  
Since
$v \in \V_\infty$, there exists a sequence $v_\ell \in \V_\ell$,
$\ell\in\N_0$,  with
  $\lim_{\ell\to\infty}\enorm[\ell]{v-v_\ell} \to 0 $  and $\limsup_{\ell \to \infty} \enorm[\ell]{v_\ell}<
  \infty$. From 
  Proposition \ref{prop:boundedness_of_energy_norm} we have 
  \begin{align*}
    \int_{D_\elm}
  |\DG^2 v_\ell|^2  \dx&+  
  \sum_{\substack{ \side \in \sidesk[\ell] \\  \side \subset
  D_\elm}} 
\int_{  \side}  h_{\ell}^{-1}  \jumpn{}{v_\ell}^2 \ds\\
&\nearrow  \int_{D_\elm}
  \abs{\DG^2 v} ^2 \dx+  
  \sum_{\substack{ \side \in \sides^+ \\  \side \subset
  D_\elm}} 
\int_{  \side}  \hG[+]^{-1}  \jumpn{}{v}^2 \ds
  \end{align*}
as $\ell \to \infty$. Let $\ell \geq k$. Thanks to
Lemma~\ref{lem:Poincare_estimate_local}  there exists
$Q_\ell\in\polspace[1](D_\elm)$ with
\begin{align*}
  \snorm[H^1(D_\elm)]{v_\ell-Q_\ell}^2 \Cleq \int_{D_\elm}
  \hG[k]^2   |\DG^2 v_\ell|^2  \dx+  
  \sum_{\substack{ \side \in \sidesk[\ell] \\  \side \subset
  D_\elm}} 
\int_{  \side} \hG[k]^2  \hG[\ell]^{-1}  \jumpn{}{v_\ell}^2 \ds\\
\nearrow  \int_{D_\elm}
  \hG[k]^2   |\DG^2 v|^2  \dx+  
  \sum_{\substack{ \side \in \sides^+ \\  \side \subset
  D_\elm}} 
\int_{  \side} \hG[k]^2  \hG[+]^{-1}  \jumpn{}{v}^2 \ds,
\end{align*}
as $\ell \to \infty$; compare also with
Proposition~\ref{prop:boundedness_of_energy_norm}.
From the
definition of $Q$ and $Q_\ell$, we have from
Proposition~\ref{prop:Friedrich_type_estimate_global} that 
\begin{align*}
  \snorm[H^1(D_\elm)]{Q_\ell-Q}^2 \Cleq
  \snorm[H^1(D_\elm)]{v_\ell-v}^2 \leq
  \snorm[H^1_0(\Omega)]{v_\ell-v}^2 \Cleq \enorm[\ell]{v_\ell -v}^2
  \to 0 
\end{align*}
as $\ell \to \infty$. Therefore,
Proposition~\ref{prop:Friedrich_type_estimate_global} implies 
$\snorm[H^1(D_\elm)]{v_\ell-Q_\ell}^2 \to
\snorm[H^1(D_\elm)]{v-Q}^2$ as $\ell \to \infty$, which finishes the proof.
\end{proof}

\subsection{Smoothing and quasi-interpolation}
\label{sec:interpol}

Before introducing the interpolation operator, we first discuss
a smoothing operator $\Cipolk[\grid]:\VG\to H_0^2(\Omega)$, $\grid\in\grids$.
To this end, following the ideas of
\cite{BrennerGudiSung2009,GeorgoulisHoustonVirtanen:09}, we 
introduce the so-called Hsieh-Clough-Tocher (HCT)
macro element constructed in~\cite{DouglasDupontPercellScott:79}.
\begin{df}[HCT element]\label{df:HTC}
  Let $\grid\in\grids$ and $\elm\in \grid$. Then the HCT nodal macro finite
  element $(\elm,\cpolspace(\elm), \cDOFs)$ is defined as follows.
  \begin{enumerate}[leftmargin=1cm, label={\alph{*})}]
  \item The local space is given by
    \begin{align*}
      \cpolspace(\elm)=\set{p \in C^1(\elm)\colon p|_{\elm_i}
    \in \polspace[4] (\elm_i), i=1, 2,3}.
    \end{align*}
    Here the three triangles $\elm_1,\elm_2$ and
    $\elm_3$ denote subtriangulation of $\elm$ obtained by connecting the vertices of $\elm$ with its
    barycenter; compare with Figure~\ref{im:element_subdivided}.
    \item The degrees of freedom $\cDOFs$ are given by (compare also with Figure~\ref{im:P2_p4_HCT})
  \begin{itemize}[leftmargin=.5cm]
  \item the function value and the gradient at the vertices of $\elm$,
  \item the function value at one interior point of each side
    $\side\in\sidesk[\grid]$, $\side\subset\partial\elm$.
  \item the normal derivative at two disctinct points in the interior
    of each  side
    $\side\in\sidesk[\grid]$, $\side\subset\partial\elm$. 
   \item the function value and the gradient at the barrycenter of $\elm$.
   \end{itemize}
\end{enumerate}
The corresponding global $H^2$-conforming finite element space is
defined as
\begin{align*}
  \widetilde \V(\grid)\definedas \{V\in C^1(\bar\Omega)\colon V|_{\elm}\in
  \cpolspace(\elm)~\text{for all}~\elm\in\grid\}
\end{align*}
and its global degrees of freedom are given by
\begin{align*}
  \cDOFs[\grid]:=\bigcup_{\elm\in \grid}\cDOFs,
\end{align*}
which is well-posed thanks to conformity of $\widetilde
\V(\grid)\subset H^2(\Omega)$.
\end{df}

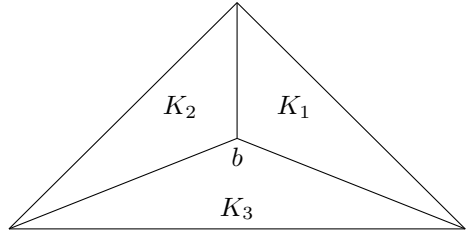
\begin{figure}[]
  \centering
  \begin{tikzpicture}[scale=.6]
  \coordinate (A) at (0,0);
  \coordinate (B) at (10,0);
  \coordinate (C) at (5,5);
  \coordinate [label=below:$b$] (b) at (5,2);

\draw (A)-- node[above]{$\elm_3$}(B);

\draw (A) -- (C) coordinate[midway] (AC);
\path  (AC) --  node[above] {$\elm_2$}  (b);

\draw (B) -- (C) coordinate[midway] (BC);
\path  (BC) --  node[above] {$\elm_1$}  (b);
\draw (A) -- (b);
\draw(B) --(b);
\draw(b) --(C);
  
\end{tikzpicture}
\caption{A macro triangle $\elm$ subdivided into three small sub triangles which share a common point $b$.}
\label{im:element_subdivided}
\end{figure}

 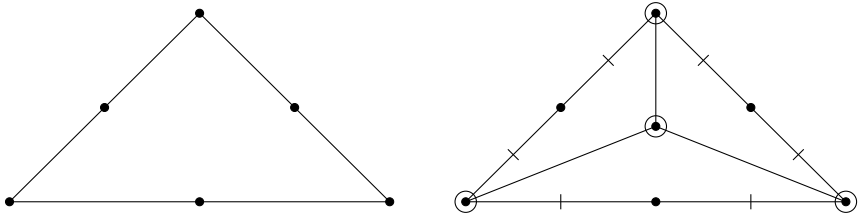
\begin{figure}
\centering
\begin{tikzpicture}[scale=1]
  \coordinate (A) at (0,0); 
  \coordinate (B) at (5,0);
  \coordinate (C) at (2.5,2.5);

  \coordinate (D) at (6,0); 
  \coordinate (E) at (11,0);
  \coordinate (F) at (8.5,2.5);

  \draw (A) --  (B) coordinate[midway] (AB); 
  \draw (A) --(C) coordinate[midway] (AC);
  \draw (B) -- (C) coordinate[midway] (BC);

  \filldraw (A) circle (1.5pt); 
  \filldraw (B) circle (1.5pt);
  \filldraw (C) circle (1.5pt);
  \filldraw (AB) circle (1.5pt);
  \filldraw (AC) circle (1.5pt);
  \filldraw (BC) circle (1.5pt);

   \draw (D) -- (E) coordinate[midway] (DE); 
   \draw (D) --(F) coordinate[midway] (DF);
   \draw (F)-- (E) coordinate[midway] (FE);
   
\filldraw let \p1 = (DE) in (\x1,1) circle (1.5pt) coordinate(b);

\draw (D)--(b);
\draw (E)--(b);
\draw (F)--(b);

\draw (D) circle (4pt);
\draw (E) circle (4pt);
\draw (F) circle (4pt);
\draw (b) circle (4pt);

 \filldraw (D) circle (1.5pt); 
  \filldraw (E) circle (1.5pt);
  \filldraw (F) circle (1.5pt);
  \filldraw (DE) circle (1.5pt);
  \filldraw (DF) circle (1.5pt);
  \filldraw (FE) circle (1.5pt);

\draw[opacity=0](D)-- (DF) coordinate[midway] (DDF);
 \draw  ($(DDF)!0.1cm!270:(D)$) -- ($(DDF)!0.1cm!90:(D)$);

\draw[opacity=0](DF)-- (F) coordinate[midway] (DFF);
 \draw  ($(DFF)!0.1cm!270:(F)$) -- ($(DFF)!0.1cm!90:(F)$);

\draw[opacity=0](F)-- (FE) coordinate[midway] (FFE);
 \draw  ($(FFE)!0.1cm!270:(F)$) -- ($(FFE)!0.1cm!90:(F)$);
 
\draw[opacity=0](E)-- (FE) coordinate[midway] (EEF);
 \draw  ($(EEF)!0.1cm!270:(E)$) -- ($(EEF)!0.1cm!90:(E)$);

\draw[opacity=0](D)-- (DE) coordinate[midway] (DDE);
 \draw  ($(DDE)!0.1cm!270:(D)$) -- ($(DDE)!0.1cm!90:(D)$);

\draw[opacity=0](DE)-- (E) coordinate[midway] (DEE);
 \draw  ($(DEE)!0.1cm!270:(E)$) -- ($(DEE)!0.1cm!90:(E)$);

\end{tikzpicture}
\caption{The Lagrange element of degree two and the corresponding macro element of degree four. Here point evaluations are denoted by small dots, (first) partial derivatives by circles and normal derivatives by lines.}
\label{im:P2_p4_HCT}
\end{figure}


Since $\polspace[2](\elm)\subset \cpolspace[4](\elm)$, we can apply
$\cDOFs[\elm]$ to $\polspace[2](\elm)$. We therefore define the  \textit{smoothing operator}  $\Cipolk[\grid]
\colon  \VG \to \widetilde \V(\grid)\subset H^2_0(\Omega)$, by setting for
all degrees of freedom $N_z\in\cDOFs[\grid]$: 
\begin{align}\label{df:enriching_operator}
  N_z(\Cipolk[\grid](v))=
  \begin{cases}
    \frac{\abs{\elm}}{\abs{\omega_k(z)}} \sum\limits_{\elm \in
      \omega_k(z)}N_z^\elm(v|_{\elm}) &\text{if }z \in \cnodes\cap\Omega \\
0 & \text{if } z\in \cnodes\cap\partial \Omega.
  \end{cases}
\end{align}
Here $\cnodes$ denotes the set of nodes $z$ associated with some
degree of freedom $N_z\in\cDOFs[\grid]$ and corresponding local degree
of freedom $N_z^\elm\in\cDOFs[\elm]$. Note that there may be different
degrees of freedom associated with one node; compare with Figure~\ref{im:P2_p4_HCT}.

\begin{lem}[$H^2_0(\Omega)$-smoothing]
\label{lem:stability_enriching}
 Let $\grid\in\grids$. The operator $\Cipolk[\grid] \colon \VG \to H^2_0(\Omega)$ defined
 in~\eqref{df:enriching_operator} satisfies for each $\elm\in\grid$ 
  \begin{align*}
    \norm[ \elm]{D^\alpha (v-\Cipolk[\grid](v))}^2
    \Cleq \int_{\sides(\neighG(\elm))}
   \abs{  \hG^{\frac32
    -\alpha}\jumpn{}{ v}}^2\ds,\qquad \alpha=0,1,2
  \end{align*}
  where the hidden constant depends only on the shape coefficient of $\grid_0$.
\end{lem}
\begin{proof}
  The proof follows from the estimates
  \cite[(2.10)-(2.12)]{BrennerGudiSung2009} together with an inverse estimate.
\end{proof}

Denoting by
$\cnodes[\elm]$ the set of points in $\elm$ associated with the degrees of freedom
$\cDOFs[\grid]$, we have $\nodes[\elm]\definedas \nodes\cap\elm\subset\cnodes[\elm]$.
This enables us to define a Cl{\'e}ment-type quasi-interpolation
$\Gipolk[\grid]:L^1(\Omega)\to L^1(\Omega)$, which is locally a left inverse of the
smoothing operator $\Cipolk[\grid]$ on $\VG$, i.e.,
\begin{align*}
    \Gipolk[\grid] \circ \Cipolk[\grid] |_{\VG}=\id|_{\VG}.
\end{align*}
To this end, we define the operator based on extensions 
of the local degrees of freedoms $\cDOFs[\elm]$
instead of $\DOFs[\elm]$.

To be more precise, for $\elm \in \grid$, let
$\set{\phi^\elm_N \colon N\in\cDOFs[\elm]}$ be the nodal basis of
$\cpolspace[4](\elm)$ and identify $\cDOFs[\elm]$ with the dual basis $
  \set{\phi^{\elm,*}_N \colon N\in\cDOFs[\elm]}\subset\cpolspace[4](\elm)$
, i.e.
\begin{align*}
  \scp[L^2(\elm)]{\phi^{\elm,*}_M}{ \phi^{\elm}_N}= M(\phi^\elm_N)=\delta_{NM}\quad N,M\in \cDOFs[\elm].
\end{align*}
Recalling Definition~\ref{df:HTC}, we have that $\cDOFs$ contains 
the point evaluation 
in the vertices and edge midpoints of $\grid$ (the Lagrange nodes $\nodes[\elm]$
of $\polspace[2](\elm)$). For $z \in\nodes[\elm]$, we denote the corresponding dual 
basis functions by
\begin{align*}
  \phi^{\elm,*}_z\in \set{\phi^{\elm,*}_N \colon
  N\in\cDOFs[\elm]}\quad\text{such that}\quad \scp[L^2(\elm)]{\phi^{\elm,*}_z}{
  v}=v(z)~\text{for all}~v\in \cpolspace[4](\elm).
\end{align*}

Extending
each local dual function by zero to a function in $L^2(\Omega)$ we
define 
\begin{align*}
\phi^{*}_z:=  \frac{1}{\abs{\omega_\grid(z)}} \sum_{\elm \in \omega_\grid(z)}
  \phi^{\elm,*}_z \in \V(\grid)^*,\quad z\in\nodes[\grid].
\end{align*}
Obviously, $\supp(\phi_z^*)\subset \omega_\grid(z)$ and
\begin{align*}
  \scp[L^2(\Omega)]{ \phi_z^*}{ v} = v(z)\qquad\text{for all}~z\in
  \nodes[\grid], v\in \widetilde \V(\grid).
\end{align*}

We define a
quasi-interpolation operator 
$\Gipolk[\grid] \colon L^1(\Omega) \to \VG$ by
\begin{align}
  \label{df:interpolation_operator_global}
    (\Gipolk[\grid] v)(z):=
  \begin{cases}
    \scp[L^2(\Omega)]{\phi_z^*}{v},\quad&\text{if}~z\in \nodes[\grid]\cap\Omega
    \\
    0, \quad&\text{if}~z\in \nodes[\grid]\cap\partial\Omega.
  \end{cases}
\end{align}

Since this definition differs from standard Cl{\'e}ment interpolation in
\cite{Clement1975:Approximation-b} only by the choice of a different
but nevertheless piecewise polynomial dual basis representation, we
obtain the following results from standard arguments; see
\cite{Clement1975:Approximation-b}.  

\begin{lem}[Quasi-interpolation onto $\VG$]
  \label{lem:properties_interpolation_operator}
  For $\grid\in\grids$ let $\Gipolk[\grid] \colon 
    L^1(\Omega)\to \VG$ be defined as
    in~\eqref{df:interpolation_operator_global}. Then we have that:
  \begin{enumerate}[leftmargin=1cm,  label*=\alph*)]
  \item \label{it:stability_interpolation}
$\Gipolk[\grid]\colon L^p(\Omega) \to L^p(\Omega)$ is a linear
    and bounded projection for all $1 \leq p \leq \infty$ and
    is stable in the following sense: If $v \in H_0^1(\Omega)$ and $\ell \in \N$, then
    \begin{align*}
      \int_{\omega_\grid^\ell(\elm)} \abs{\nabla \Gipolk[\grid] v}^2 \dx \Cleq
      \int_{\omega_\grid^{\ell+1}(\elm)} \abs{\nabla v}^2
      \dx\qquad\text{for all}~\elm\in\grid.
    \end{align*}
  \item \label{it:interpol_is_projection}
 $\Gipolk[\grid] v \in \VG $ for all $v \in L^1(\Omega)$,
  \item \label{it:interpol_preserves_polynomials_locally}
$\Gipolk[\grid] v|_\elm =v|_\elm$ on $\elm \in \grid$ with $\elm
    \cap \partial \Omega = \emptyset$ if
    $v|_{\omega_\grid(\elm)} \in \polspace[2](\neighG(\elm))\cap C(\omega_\grid(\elm))$,
  \item \label{it:interpol_is_left_inverse_of_enriching}
    $\Gipolk[\grid](\Cipolk[\grid] v)|_\elm=v|_\elm$ on $\elm \in
    \grid$ if 
    $v|_{\omega_\grid(\elm)} \in \polspace[2](\neighG(\elm))\cap
    C(\omega_\grid(\elm))$ and  $\elm
    \cap \partial \Omega = \emptyset$ or
    $v|_{\partial\Omega\cap\elm}=0$. Here
    $\Cipolk[\grid] \colon \VG \to H^2_0(\Omega)$ is the enriching
    operator defined in~\eqref{df:enriching_operator}.
  \end{enumerate}
\end{lem}

We remark that, in principle, one
can also resort to a Scott-Zhang-type quasi interpolation
\cite{ScottZhang:90}. However, this complicates the construction of
$\Gipolk[\grid]$, since a dual basis, bi-orthogonal to the nodal basis of traces of functions in
$\cpolspace(\elm)$, needs to be constructed on faces of  boundary elements
. The price we have to pay for the simpler
construction is that the set of integration needs to be slightly increased in the right hand side of the
following stability estimate. We are
particularly interested in the interplay of different refinement
levels related to the sequence $\{\gridk\}_{k\in\N_0}$ of meshes
produced by the \ACIPGM. To simplify notation, we  again replace
subscripts $\gridk$ by $k$, e.g. we write $\Gipolk$ instead of $\Gipolk[\gridk]$.

\begin{lem}[Stability of $\Gipolk $]
\label{lem:stability_interpolation}
  Let $v \in \V_\ell $ for some $\ell \in \N_0 \cup
  \set{\infty}$. Then, for all $\elm  \in \gridk$, $k\le \ell$, we have 
  \begin{align*}
    \int_\elm \abs{D^2 \Gipolk v}^2 \dx&+ \int_{\partial \elm}
    \hG[k]^{-1} \jumpn{}{ \Gipolk v}^2 \ds \\
&\Cleq \int_{\omega^3_k(\elm) } \abs{\DG^2 v}^2 \dx+
  \sum_{\substack{ \side \in \sidesk[\ell] \\  \side \subset
 \omega^3_{k}(\elm) }} 
\int_{  \side}  \hG[\ell]^{-1}  \jumpn{}{v}^2 \ds,
  \end{align*}
where 
 $\sides_\ell:=\sides^+$ and $\hG[\ell]:=\hG[+]$, when $\ell
=\infty$. In particular, we have $\enorm[k]{\Gipolk v} \Cleq
\enorm[\ell]{v}$. 

Moreover, for $w\in H_0^2(\Omega)$, we have $\enorm[k]{\Gipolk w} \Cleq \norm[\Omega]{D^2
  w}$.
\end{lem}
\begin{proof}
  %
  Let $\ell < \infty$ and assume that $\elm \in \gridk$ such
that ${\omega_k^2(\elm)}
\cap \partial \Omega= \emptyset$. Let  $Q$ be the linear
polynomial from Lemma~\ref{lem:Poincare_estimate_local}
with  $\grid=\gridk$, $\grid^\star=\grid_\ell$, and
$D_\elm=\omega_k(\elm)$. Then Lemma~\ref{lem:properties_interpolation_operator}\ref{it:stability_interpolation} and
\ref{it:interpol_preserves_polynomials_locally} yield
\begin{align*}
  \int_\elm \abs{D^2 \Gipolk v}^2 \dx&=\int_\elm \abs{D^2 \Gipolk
                                       (v-Q)}^2 \dx 
\Cleq \int_\elm \hG[k]^{-2}\abs{\nabla \Gipolk
                                       (v-Q)}^2 \dx
  \\
&\Cleq  \int_{\omega_k(\elm)} \hG[k]^{-2}\abs{\nabla 
                                                         (v-Q)}^2 \dx
  \\
&\Cleq  \int_{\omega_k(\elm)}
   \abs{\DG^2 v}^2  \dx+  
  \sum_{\substack{ \side \in \sidesk[\ell] \\  \side \subset
  \omega_k(\elm)}} 
\int_{  \side}  \hG[\ell]^{-1}  \jumpn{}{v}^2 \ds.
\end{align*}

In order to bound the jump terms, let $Q$ be the linear polynomial from
Lemma~\ref{lem:Poincare_estimate_local} with  $\grid=\gridk$,
$\V(\grid^\star)=\V_\ell$, and $D_\elm=\omega^2_k(\elm)$. We observe
that $\nabla Q\equiv 
const$ and hence does not jump across interelement
boundaries. Consequently, using  Lemma~\ref{lem:properties_interpolation_operator}\ref{it:stability_interpolation} and
\ref{it:interpol_preserves_polynomials_locally}, together with
a scaled trace theorem and  inverse estimates, we obtain
\begin{align*}
  \int_{\partial \elm} \hG[k]^{-1}\jumpn{}{\Gipolk v}^2 \ds&=
  \int_{\partial
    \elm} \hG[k]^{-1}
  \jumpn{}{\Gipolk( v-Q)}^2 \ds
  \\
  & \Cleq 
  \int_{\omega_k(\elm)} \hG[k]^{-2}\abs{\nabla \Gipolk( v-Q)}^2 \dx
  \Cleq  \hG[\elm]^{-2} \int_{\omega^2_k(\elm)} \abs{ \nabla(v-Q)}^2 \dx
  \\
  &\Cleq  \int_{\omega^2_k(\elm)}
  \abs{\DG^2 v}^2  \dx+  
  \sum_{\substack{ \side \in \sidesk[\ell] \\  \side \subset
  \omega^2_k(\elm)}} 
\int_{  \side}  \hG[\ell]^{-1}  \jumpn{}{v}^2 \ds,
\end{align*}
where we also used $\bigcup\set{\omega_k(\side) \colon \side
  \subset \partial \elm} \subset \omega_k(\elm)$ and
Lemma~\ref{lem:Poincare_estimate_local}\eqref{eq:Poincare}.

If ${\omega_k^2(\elm)} \cap \partial \Omega \not = \emptyset$, then
there exists a side $\side \in\sidesk[k]$ with $\side \subset{\omega_k^3(\elm)} \cap \partial \Omega$. Now applying 
\eqref{eq:Friedrichs} instead of~\eqref{eq:Poincare} the desired
assertion follows similarly as above.

For $\ell =\infty$ we replace Lemma~\ref{lem:Poincare_estimate_local} by Lemma~\ref{lem:poincare_local_limit} and proceed as before.

For $w\in H_0^2(\Omega)$ the estimate follows by analogous arguments
replacing Lemma~\ref{lem:Poincare_estimate_local} by the classical
Poincar\'{e}-Friedrichs inequality for functions in $H^2_0(\Omega)$ 
together with scaling arguments.
\end{proof}

\begin{cor}
\label{cor:interpolation_estimate}
  Let $v \in \V_\ell$, $\ell \in \N_0
  \cup \set{\infty}$ and $\elm \in \gridk$ for some $k\le\ell$. Then
\begin{align*}
    \int_\elm \abs{D^2 \Gipolk v -\DG^2 v }^2 \dx&+ \int_{\partial \elm}
    \hG[k]^{-1} \jumpn{}{(\Gipolk v-v)}^2 \ds \\
&\Cleq \int_{\omega^3_k(\elm)} \abs{\DG^2 v}^2 \dx+
 \sum_{\substack{ \side \in \sidesk[\ell] \\  \side \subset
  \omega^3_{k}(\elm)}} 
\int_{  \side}  \hG[\ell]^{-1}  \jumpn{}{v}^2 \ds,
  \end{align*}
where we write $\sidesk[\ell]\definedas\sides^+$ and $\hG[\ell]\definedas\hG[+]$ if $\ell
=\infty$ as in Lemma~\ref{lem:stability_interpolation}.
\end{cor}
The next corollary states the convergence of the interpolation.

\begin{cor}
\label{cor:convergence_interpolation}
  Let $v \in \V_\infty$, then
  $ \enorm[k]{\Gipolk v-v} \to 0$ as $k \to \infty$.
\end{cor}

\begin{proof}
  Thanks to the definition of $ \V_\infty$ there exist a sequence  $\set{v_k}_{k
  \in \N_0}$, $v_k \in \V_k$ with $\enorm[k]{v_k - v} \to 0$ as $ k \to
\infty$. Consequently, the claim follows from stability and invariance of the interpolation
operator. 
 \end{proof}

\subsection{Proof of Lemma \ref{lem:limit_space_is_Hilbert}}
\label{sec:limit_space_is_Hilbert}
From the definition of the space $\V_\infty$ it is clear that we only need to check for completeness in order to conclude the assertion.

Let $\{v^\ell\}_{\ell\in\N_0}$ be a Cauchy sequence in
$(\V_\infty,\enorm[\infty]{\cdot})$. Note that, the broken
\Poincare-Friedrichs inequality
(Proposition~\ref{prop:Friedrich_type_estimate_global}) is inherited
to $\V_\infty$, and thus  we have 
$v^\ell\to v \in  H^1_0(\Omega)$ as $\ell\to\infty$ and it remains to
prove that $v\in \V_\infty$. 
Using norm equivalence on finite dimensional spaces,
we readily conclude that $v|_\elm \in \polspace[2](\elm)$ for all $\elm \in
\grid^+$. Moreover, thanks to 
Propositions~\ref{prop:total_variation_of_gradient_bounded_by_energy}
and~\ref{prop:boundedness_of_energy_norm}, we have that 
$\{\nabla v^\ell\}_{\ell\in\N_0}$ is also a Cauchy sequence in
$BV(\Omega)^2$, i.e. $\nabla v\in BV(\Omega)^2$. Consequently, 
$\nabla v\in BV(\Omega)^2$ has an $L^1$ trace on sides 
$\side \in \sidesk[k]$, $k \in \N_0$; see e.g.~\cite[Theorem
3.88]{Ambrosio:2000}.

\boxed{1} We shall first deal with the jumps of the normal
derivatives. To this end, we first observe that for $k\in\N_0$,
$\{v^\ell\}_{\ell\in\N_0}$ is also a Cauchy sequence with respect to
the $\enorm[k]{\cdot}$-norm
(Proposition~\ref{prop:boundedness_of_energy_norm}) and thus
uniqueness of limits imply on $\Gamma_k=\Gamma(\gridk)$ that 
$ \nabla v^\ell\to  \nabla v$ in $L^2(\Gamma_k)$ as
$\ell\to\infty$ in the sense of traces. Moreover, we have that
$\int_{\sides_k}h_k^{-1}\jumpn{}{v}^2\ds$ is uniformly bounded.   
Let $\epsilon>0$ arbitrary fixed, then there exists $L=L(\epsilon)$, such that
$\enorm[k]{v^\ell-v^j}\le\enorm[\infty]{v^\ell-v^j}\le \epsilon$ for all $j, \ell\ge
L$. Thanks to
Proposition~\ref{prop:boundedness_of_energy_norm}, there exists
$K=K(\epsilon,L)$ such that for all $m\ge k\ge K$, we have
\begin{align}\label{k-jump<eps}
  \int_{\sides_m\setminus\sides_k^+}h_+^{-1} \jumpn{}{v^L}^2\ds\le   \epsilon^2.
\end{align}
In particular, for $m=k\ge K$, we have
\begin{align*}
    \int_{\sidesk} \hG[k]^{-1} \jumpn{}{v}^2 \ds &=
    \int_{\sidesk\setminus\sidesk^+} \hG[k]^{-1} \jumpn{}{v}^2
    \ds+\int_{\sidesk^+} \hG[k]^{-1} \jumpn{}{v}^2 \ds
    \\
    &=\lim_{\ell\to\infty}\int_{\sidesk\setminus\sidesk^+} \hG[k]^{-1}
    \jumpn{}{v^\ell}^2 \ds+\int_{\sidesk^+} \hG[k]^{-1} \jumpn{}{v}^2
    \ds
\end{align*}
and 
\begin{align*}
  \int_{\sidesk\setminus\sidesk^+} \hG[k]^{-1}
    \jumpn{}{v^\ell}^2 \ds&\le 
    2\enorm[k]{v^\ell-v^L}^2+2\int_{\sidesk\setminus\sidesk^+}
    \hG[k]^{-1} \jumpn{}{v^L}^2 \ds
      \le 4\epsilon^2
\end{align*}
provided $\ell\ge L$,
which together with Proposition~\ref{prop:boundedness_of_energy_norm}
leads to
\begin{align}\label{conv:jumps_CS}
  \int_{\sidesk} \hG[k]^{-1}
  \jumpn{}{v}^2 \ds \to
  \int_{\sides^+} \hG[+]^{-1}
  \jumpn{}{v}^2 \ds \qquad\text{as}~ k\to\infty,
\end{align}
since $\epsilon>0$ was arbitrary.

  \boxed{2}  In order to prove  $v|_{\Omega^-} \in
  H^2_{\partial \Omega \cap \partial \Omega^-}(\Omega^-)$ we need to
  show that $v$ is a restriction of a 
  $H_0^2(\Omega)$-function. To this end, thanks to Corollary~\ref{cor:convergence_interpolation},
  there exists $\{m_\ell\}_{\ell\in\N_0}\subset\N_0$
  such that $\enorm[m_\ell]{v^\ell-v^\ell_{m_\ell}}\leq \frac 1 \ell $
  for $v_{m_\ell}\definedas\Gipolk[m_\ell] v^\ell \in
  \V_{m_\ell}$, where $\Gipolk[m_\ell] v^\ell $ is the
    interpolant from \eqref{df:interpolation_operator_global}
    with respect to $\grid_{m_\ell}$.
  Consequently, since $\{v^\ell\}$ is a Cauchy sequence and thus
  bounded, we infer from
  Proposition~\ref{prop:boundedness_of_energy_norm} that
  \begin{align*}
     \enorm[m_\ell]{v^\ell_{m_\ell}} &\leq   \enorm[m_\ell]{v^\ell_{m_\ell}-v^\ell}
                                      + \enorm[\infty]{v^\ell}
                                     \leq \frac{1}{\ell} +\enorm[\infty]{v^\ell} ,  
  \end{align*}
  i.e., the uniform boundedness of $
  \enorm[m_\ell]{v^\ell_{m_\ell}}$. 
 We now apply the smoothing
  operator defined in~\eqref{df:enriching_operator} to 
  $v^\ell_{m_\ell}\in \V_{m_\ell}$ together with Lemma~\ref{lem:stability_enriching} ($\alpha=2$) and obtain  
\begin{align*}
  \norm[\Omega]{ D^2 \Cipolk[m_\ell]( v^\ell_{m_\ell})} \Cleq
  \norm[\Omega]{\DG^2 (\Cipolk[m_\ell] 
  (v^\ell_{m_\ell}) -v^\ell_{m_\ell})} +\norm[\Omega] {\DG^2 v^\ell_{m_\ell}
 }  \Cleq  \enorm[m_\ell]{v^\ell_{m_\ell}}.
\end{align*}
Hence, there exists $w
\in H^2_0(\Omega)$ such that, for a not relabelled subsequence
\begin{align} \label{conv:enrich_FEM_function}
  \Cipolk[m_\ell] (v^\ell_{m_\ell}) \wconv w \quad \text{weakly in }H^2_0(\Omega),\qquad\text{as}~\ell\to\infty.
\end{align}
Again from Lemma~\ref{lem:stability_enriching} (for $\alpha=0$) and the scaled trace theorem, we have that 
\begin{align}\label{eq:Eml->vml}
  \begin{aligned}
    \norm[\Omega^{-}_{m_\ell}]{\Cipolk[m_\ell] (v^\ell_{m_\ell})
      -v^\ell_{m_\ell}}^2 &\Cleq \int_{{\sidesk[m_\ell] ^{2-}}}
    \hG[m_\ell]^3\jumpn{}{ v^\ell_{m_\ell}}^2 \ds
    \\
    &\Cleq
    \norm[L^\infty(\Omega)]{h_{m_\ell}\chi_{\Omega^{2-}_{m_\ell}}}^4
    \enorm[m_\ell]{v^\ell_{m_\ell}}^2,
  \end{aligned}
\end{align}
where we used $\norm[L^\infty({\sidesk[m_\ell] ^{2-}})]{h_{m_\ell}} \Cleq
\big\|{h_{m_\ell}\chi_{\Omega^{2-}_{m_\ell}}}\big\|_{L^\infty(\Omega)}$. Applying
Lemma~\ref{lem:Omega_star_and_h_vanishing}, the last term vanishes as
$\ell \to \infty$. Thanks to a
Poincar\'e-Friedrichs' inequality and
Proposition~\ref{prop:total_variation_of_gradient_bounded_by_energy},
we have 
\begin{align*}
  \norm[\Omega]{v - v^\ell_{m_\ell}} \Cleq \snorm[H^1_0(\Omega)]{v -
  v^\ell_{m_\ell}} \Cleq \snorm[H^1_0(\Omega)]{v-v^\ell} + 
  \enorm[m_\ell]{v^\ell - v^\ell_{m_\ell}} \to 0 \qquad \text{as}~\ell \to \infty.
 \end{align*}
and thus $v|_{\Omega^-}=w|_{\Omega^-}$ from $\Omega^-
\subset \Omega^{2-}_{m_\ell}$, i.e., $v|_{\Omega^-} \in
H^2_{\partial \Omega \cap \partial \Omega^-}(\Omega^-)$.
Therefore, we can use the definition~\eqref{df:D2pw} of the piecewise
Hessian also for $v$.

\boxed{3} We shall use the construction of \step{2} in order to show
that $v$ is the limit of $v^\ell$,
i.e., that $\enorm[\infty]{v-v^\ell}\to0$ as $\ell\to\infty$.
To this end, arguing similar as for~\eqref{conv:jumps_CS} we have 
\begin{align}\label{jump(v-v^ell)->0}
  \int_{\sides^+}h_+^{-1}\jumpn{}{(v-v^\ell)}^2\ds\to 0
  \qquad\text{as}~\ell\to \infty.
\end{align}
It therefore remains to prove that $\norm{\DG^2 v-\DG^2 v^\ell}\to 0$
as $\ell\to 0$. Since the Cauchy sequence property implies that
$\norm[L^2(\Omega)]{\DG^2 v^\ell- \vec d}\to 0$ for some $\vec d\in
L^2(\Omega)^{2\times 2}$ as $\ell\to\infty$, it thus suffices to prove
$\DG^2 v=\vec d$.

To this end, we first conclude as
for~\eqref{eq:Eml->vml} from Lemma~\ref{lem:stability_enriching} (but
this time
for $\alpha=2$)
that
\begin{align*}
  \norm[\Omega^{-}_{k}]{D^2\Cipolk[m_\ell] (v^\ell_{m_\ell})
      -\DG^2v^\ell_{m_\ell}}^2 &\Cleq \int_{{\sidesk[m_\ell]\setminus\sidesk^{2+}}}
    \hG[m_\ell]^{-1}\jumpn{}{ v^\ell_{m_\ell}}^2 \ds\le 4\epsilon^2+ \frac2{\ell^2},
\end{align*}
where $m_\ell\ge k\ge K$ as for \eqref{k-jump<eps} as well as
$\enorm[m_\ell]{v^\ell_{m_\ell}-v^\ell}<1/\ell$, $\ell\ge L$, and
Proposition~\ref{prop:boundedness_of_energy_norm} in the last step.
We apply this now to the distributional Hessian 
\begin{align*}
  \langle D^2 v^\ell,\vec\varphi \rangle &= \int_{\Omega_k^-} \DG^2
  v^\ell \colon \vec\varphi  \dx+\int_{\Omega_k^+} \DG^2
  v^\ell \colon \vec\varphi  \dx
  -\int_{\sidesk^+}\jumpn{}{v^\ell}\vec\varphi \normal\cdot\normal\ds\\
  &= \int_{\Omega_k^-} \DG^2\Cipolk[m_\ell] (v^\ell_{m_\ell})
  \colon \vec\varphi  \dx - \int_{\Omega_k^-} (\DG^2\Cipolk[m_\ell] (v^\ell_{m_\ell})-\DG^2v^\ell)
    \colon \vec\varphi \dx
  \\
  &\quad+\int_{\Omega_k^+} \DG^2
  v^\ell \colon \vec\varphi  \dx
  -\int_{\sidesk^+}\jumpn{}{v^\ell}\vec\varphi \normal\cdot\normal\ds
\end{align*}
with $\vec\varphi\in
C_0^\infty(\Omega)^{2\times 2}$. In fact in combination
with~\eqref{jump(v-v^ell)->0} and \eqref{conv:enrich_FEM_function}, we
obtain
\begin{gather*}
  \int_{\Omega_k^-} \DG^2\Cipolk[m_\ell] (v^\ell_{m_\ell})
  \colon \vec\varphi  \dx\to \int_{\Omega_k^-} D^2w
  \colon \vec\varphi  \dx\quad\text{as}~\ell\to\infty
  \\
  \lim_{\ell\to\infty}\abs{\int_{\Omega_k^-} (\DG^2\Cipolk[m_\ell] (v^\ell_{m_\ell})-\DG^2v^\ell)
    \colon \vec\varphi  \dx}\Cleq \epsilon\norm[L^2(\Omega)]{\vec\varphi}
  \intertext{and}
  \int_{\Omega_k^+} \DG^2
  v^\ell \colon \vec\varphi  \dx
  -\int_{\sidesk^+}\jumpn{}{v^\ell}\vec\varphi
  \normal\cdot\normal\ds\to \int_{\Omega_k^+} \DG^2
  v\colon \vec\varphi  \dx
  -\int_{\sidesk^+}\jumpn{}{v}\vec\varphi \normal\cdot\normal\ds
\end{gather*}
as $\ell\to\infty$, where we have used 
strong covergence $ v^\ell|_{\Omega_k^+}\to
v|_{\Omega_k^+}$ in $\P_2(\gridk^+)$ for the last estimate. We thus have
for all $\vec\varphi\in
C_0^\infty(\Omega)^{2\times 2}$
\begin{align*}
  \abs{\int_\Omega(\chi_{\Omega_k^-} D^2w+\chi_{\Omega_k^+}\DG^2v -
  \vec d)\colon\vec\varphi  \dx}\Cleq
  \epsilon  \norm[L^2(\Omega)]{\vec\varphi}
\end{align*}
Now using uniform integrability as $k\to\infty$ and recalling that
$\epsilon>0$ was arbitrary, we conclude the assertion since
$\DG^2v|_{\Omega^-}=\DG^2 w|_{\Omega^-}$.

\boxed{4} We conclude by showing that that for  $v_k := \Gipolk w \in \V_k$, $k\in\N_0$,
we have $\enorm[k]{v-v_k} \to 0 $ as $k \to \infty$, and
$\limsup_{k \to \infty}\enorm[k]{v_k}<\infty$;
here $w \in H^2_0(\Omega)$ is the function defined
in~\eqref{conv:enrich_FEM_function}.
The uniform boundedness follows since from Lemma~\ref{lem:stability_interpolation}, we
have
\begin{align*}
  \enorm[k]{v_k} \Cleq \sum_{\elm \in \gridk}
\int_{\omega_k^3(\elm)} \abs{D^2w}^2 \dx \Cleq
  \norm[H^2_0(\Omega)]{w}< \infty.
\end{align*}

We split 
$\enorm[k]{v-v_k}^2$ according to
$\gridk=\gridk^{2-} \cup\gridk^{2+} $ and consider the corresponding terms
  separately. 
On the set $\grid_k^{2-}$ we use the density of $H^3_0(\Omega)$ in
$H^2_0(\Omega)$ and choose for arbitrarily fixed $\epsilon>0$ some $w_\epsilon \in H^3_0(\Omega)$ such that
$\norm[H^2(\Omega^-)]{w-w_\epsilon} \leq
\norm[H^2(\Omega)]{w-w_\epsilon} <\epsilon$. Thanks to the triangle
inequality and the stability of $ \Gipolk$ (Lemma~\ref{lem:stability_interpolation}), we have
\begin{align}\label{eq:|v-v_k|Tk-}
  \begin{aligned}
   &\sum_{\elm \in \gridk^{2-}} \bigg[
          \int_\elm  \abs{D^2 \Gipolk w -\DG^2 v }^2 \dx+\int_{\partial \elm}
          \hG[k]^{-1} \jumpn{}{(\Gipolk w-v)}^2 \ds \bigg]
\\
&\quad  \Cleq \sum_{\elm \in \gridk^{2-}} \bigg[\int_\elm \abs{D^2 \Gipolk( w
  -w_\epsilon) }^2 +  \abs{D^2 (\Gipolk w_\epsilon
  -w_\epsilon) }^2 +
\abs{D^2 ( w_\epsilon
  -v) }^2 \dx  \\
& \qquad\qquad\quad   +\int_{\partial\elm} h_k^{-1}\jumpn{}{\Gipolk(w-w_\epsilon)} ^2
  +h_k^{-1}\jumpn{}{\Gipolk w_\epsilon} ^2  + h_k^{-1}\jumpn{}{v} ^2  \ds\bigg] \\
& \quad \Cleq \int_{N_k^3(\grid_k^{2-} ) }  \abs{D^2 (w-w_\epsilon)}^2 
\dx  +  \int_{\grid_k^{2-}}  \abs{D^2 ( w_\epsilon
  -v) }^2 + \abs{D^2(\Gipolk w_\epsilon
  -w_\epsilon) }^2 \dx\\
& \qquad +\sum_{\elm \in \grid_k^{2-}} 
  \int_{\partial\elm}  h_k^{-1}\big[
  \jumpn{}{\Gipolk w_\epsilon} ^2  + \jumpn{}{v} ^2  \big] \ds.
  \end{aligned}
\end{align}
In order to bound the terms concerning the interpolation operator, we 
employ a scaled trace theorem together with Lemma~\ref{lem:properties_interpolation_operator}\ref{it:stability_interpolation} and
\ref{it:interpol_preserves_polynomials_locally} to obtain 
\begin{align}\label{eq:|v-v_k|Tk-a}
  \begin{aligned}
    \sum_{\elm\in \gridk^{2-}} & \int_\elm \abs{D^2(\Gipolk w_\epsilon
      -w_\epsilon) }^2 + \int_{\partial \elm}\hG[k]^{-1}
    \jumpn{}{\Gipolk w_{\epsilon} }^2
    \ds \\
    &\leq 2\sum_{\elm\in \gridk^{2-}} \int_\elm \abs{D^2 \Gipolk(
      w_{\epsilon} -Q_\elm ) }^2 +\abs{D^2 ( w_{\epsilon} -Q _\elm)
    }^2 \dx
    \\
    &\quad+ \sum_{\elm\in \gridk^{2-}} \int_{\partial \elm}\hG[k]^{-1}
    \jumpn{}{\Gipolk ( w_{\epsilon} -Q_\elm) }^2
    \ds \\
    &\Cleq \sum_{\elm\in \gridk^{2-}} \int_{\omega^3_k(\elm)}
    \hG[k]^{-2}\abs{\nabla ( w_{\epsilon} -Q_\elm ) }^2 +\abs{\DG^2 (
      w_{\epsilon} -Q _\elm) }^2 \dx
    \\
    &\Cleq \int_{N^3_k(\grid_k^{2-})} \hG[k]^2 \sum_{|\alpha| =3}
    \abs{D^\alpha w_\epsilon}^2 \Cleq \norm[L^\infty(\Omega)]{\hG[k]
      \chi_{\Omega^{5-}_k}}^2 \int_{\Omega} \sum_{\abs{\alpha}=3}
    \abs{D^\alpha w_{\epsilon}}^2 \dx.
  \end{aligned}
\end{align}
Here, we have used the Bramble-Hilbert
Lemma (\cite{DupontScott:80}) for suitable chosen $Q_\elm\in\P_2(\omega_k^3(\elm))$,
$\elm\in\gridk$ in the penultimate estimate as well as $\Omega(N^3_k(\grid_k^{2-})\subset
\Omega^{5-}_k$ and 
the finite overlap of neighbourhoods in the last step. Thanks to
Lemma~\ref{lem:Omega_star_and_h_vanishing} the last term vanishes as $k\to\infty$.

Recalling $v|_{\Omega^-}=w|_{\Omega^-}$, we conclude from
Lemma~\ref{lem:Omega_star_and_h_vanishing} from  the uniform integrability of $\DG^2 v$ and $D^2
w_\epsilon$ that 
\begin{align}\label{eq:|v-v_k|Tk-b}
  \begin{aligned}
    \lim_{k\to\infty}\sum_{\elm \in \gridk^{2-}} \int_\elm \abs{D^2 
      w_\epsilon -\DG^2v }^2 \dx &+ \int_{N_k^3(\grid_k^{2-} ) } \abs{D^2
      (w-w_\epsilon)}^2
    \dx \\
    &\lesssim  \norm[H^2(\Omega)]{w_\epsilon -w}^2\le \epsilon^2,
  \end{aligned}
\end{align}
thanks to the finite overlap of neighbourhoods.

For the remaining jump term in \eqref{eq:|v-v_k|Tk-}, we infer that 
\begin{multline*}
  \sum_{\elm \in \gridk^{2-}} \int_{\partial \elm} \hG[k]^{-1} \jumpn{}{ v
  }^2
  \ds =\sum_{\elm \in \gridk} \int_{\partial \elm} \hG[k]^{-1} \jumpn{}{ v
  }^2\ds-\sum_{\elm \in \gridk^{2+}} \int_{\partial \elm} \hG[+]^{-1} \jumpn{}{ v
  }^2 \ds
  \\
  \to \sum_{\elm \in \grid^+} \int_{\partial \elm} \hG[+]^{-1} \jumpn{}{ v
  }^2\ds-\sum_{\elm \in \grid^{+}} \int_{\partial \elm} \hG[+]^{-1} \jumpn{}{ v
  }^2 \ds =0
\end{multline*}
as $k \to \infty$, thanks to
\eqref{conv:jumps_CS} and Lemma~\ref{lem:Omega_star_and_h_vanishing}.
Inserting this,~\eqref{eq:|v-v_k|Tk-a} and~\eqref{eq:|v-v_k|Tk-b}
into~\eqref{eq:|v-v_k|Tk-}, and recalling that that $\epsilon>0$ was chosen
arbitrary, we have proved 
\begin{align}\label{est:V_infty_is_Hilbert_S3_2}
  \lim_{k\to\infty}\sum_{\elm \in \gridk^{2-}} \bigg[
          \int_\elm  \abs{D^2 \Gipolk w -\DG^2 v }^2 \dx+\int_{\partial \elm}
          \hG[k]^{-1} \jumpn{}{(\Gipolk w-v)}^2 \ds \bigg] =0.
\end{align}

Let now $\elm\in \gridk^{2+}$. Then we have for all $m_\ell\ge k$
that $\gridk^{+}\subset\grid_{m_\ell}^{+}$ and thus $v_{m_\ell}^\ell|_{\omega_k(\elm)} \in \polspace[2](\neighk(\elm))\cap
C(\omega_k(\elm))$ (see
step \boxed{2} for the definition of $ v_{m_\ell}^\ell$ and
$m_\ell$). Therefore, Lemma~\ref{lem:properties_interpolation_operator}\ref{it:interpol_is_left_inverse_of_enriching} 
implies
\begin{align*}
  v_k=\Gipolk w\leftarrow \Gipolk
  \Cipolk[m_\ell]v_{m_\ell}^\ell=\Gipolk
  \Cipolk v_{m_\ell}^\ell =v_{m_\ell}^\ell\rightarrow v \quad \text{in}~\polspace[2](\elm)
\end{align*}
as $\ell\to\infty$.
Consequently, for all $k\in\N_0$, we have
\begin{align*}
  \sum_{\elm\in\gridk^{2+}} 
  \int_\elm  \abs{D^2 v_k -D^2 v }^2 \dx+\int_{\partial \elm}
  \hG[k]^{-1} \jumpn{}{(v_k-v)}^2 \ds =0.
\end{align*}
Combining this with~\eqref{est:V_infty_is_Hilbert_S3_2} we have
constructed a sequence $\{v_k\}_{k\in \N_0}$ with $v_k=\Gipolk w\in\V_k$ such
that that $\enorm[k]{v_k-v}^2\to 0$ as $k \to \infty$.  This proves
$v\in\V_\infty$.

Overall, we have thus showed that $\lim_{\ell \to \infty} v^\ell = v \in
\V_\infty$, which concludes the proof. \qed

\subsection{Proof of Theorem \ref{thm:u_k_to_u_infty}}
To identify a candidate for the limit of the sequence $\{u_k\}_{k\in\N_0}$ of discrete
approximations computed by the \ACIPGM, we employ
Proposition~\ref{prop:Friedrich_type_estimate_global}  and
\eqref{est:uniform_stab_solution}, and conclude  that
\begin{align}
  \label{conv:u_k_to_ovu_wH1}
  u_{k_j} \wconv \ovu \qquad\text{weakly in}~H^1_0(\Omega)\quad\text{as}~j\to\infty
\end{align}
for some subsequence $\{k_j\}_{j\in\N_0}\subset\{k\}_{k\in\N_0}$ and
$\ovu\in H_0^1(\Omega)$. In the following, we shall see that
in fact $u_\infty=\overline u_\infty \in \V_\infty$. Thus
$\{u_k\}_{k\in\N_0}$ has only one weak accumulation point and the whole
sequence converges. Finally we shall
conclude the section with proving  the strong convergence $\lim_{k \to \infty}\enorm[k]{u_k
  -u_\infty} = 0$ claimed in Theorem 
\ref{thm:u_k_to_u_infty}.

\begin{lem}\label{lem:u_infty_in_limit_space}
  We have $\ovu \in \V_\infty$.
  
\end{lem}
\begin{proof}
  \boxed{1}  Thanks to the uniform boundedness
\eqref{est:uniform_stab_solution}  of $\enorm[k_j]{u_{k_j}}$, we conclude
with  Propositions
\ref{prop:Friedrich_type_estimate_global} and 
\ref{prop:total_variation_of_gradient_bounded_by_energy} that 
\begin{align}\label{eq:uk->ovuBV}
  \nabla u_{k_j} \wconv^* \nabla \ovu\qquad\text{weakly* in}~BV(\Omega)^2\quad\text{as}~j\to\infty;
\end{align}
compare also with~\cite[Theorem~3.23]{Ambrosio:2000}.
Moreover, Lemma~\ref{lem:stability_enriching} ($\alpha=2$) yields for
the smoothing operator from~\eqref{df:enriching_operator} that 
\begin{align*}
  \norm[\Omega]{D^2\Cipolk[k_j] (u_{k_j})} \leq
  \norm[\Omega]{\DG^2 (\Cipolk[k_j] (u_{k_j}) -u_{k_j})}+\norm[\Omega]{\DG^2 u_{k_j}}
  \Cleq \enorm[k_j]{u_{k_j}}.
\end{align*}
We thus have 
\begin{align}\label{eq:w}
  \Cipolk[k_j] (u_{k_j}) \wconv w \quad \text{weakly in } H^2_0(\Omega)\quad\text{as}~j\to\infty
\end{align}
for a not relabelled subsequence.  Arguing as in step \boxed{2} in the
proof of Lemma~\ref{lem:limit_space_is_Hilbert}, we obtain, that
$\norm[\Omega^{2-}_{k_j}]{\Cipolk[k_j]
  (u_{k_j})-u_{k_j}} \to 0$ as $j \to \infty$ and
thus~\eqref{eq:uk->ovuBV} implies
\begin{align*}
  \ovu |_{\Omega^-}=w|_{\Omega^-} \in H^2_{\partial \Omega \cap
  \partial \Omega^-}(\Omega^-). 
\end{align*}

\boxed{2}~For $w$ from~\eqref{eq:w}, defining
\begin{align*}
  v_k:=\Gipolk w \in \V_k,
\end{align*}
we have by Lemma~\ref{lem:stability_interpolation} that
$\enorm[k]{v_k}\Cleq\norm{D^2 w}<\infty$. Therefore, in order to
conclude the proof, it remains to
show that $\enorm[k]{v_k -\ovu} \to 0$ as $k \to \infty$. In order to
see this, we observe 
that the weak
convergence~\eqref{conv:u_k_to_ovu_wH1} implies
strong convergence of the restrictions $u_{k_j}|_\elm$ 
in the finite
dimensional $\polspace[2](\elm)$ 
and thus $\ovu |_{\elm} \in \polspace[2](\elm)$.
Moreover,  thanks to
Lemma~\ref{lem:properties_interpolation_operator}\ref{it:interpol_is_left_inverse_of_enriching},
we have 
$\Gipolk[k]\Cipolk[k_j]u_{k_j}|_\elm=\Gipolk[k_j]\Cipolk[k_j]u_{k_j}|_\elm$
for $K\in \gridk^{1+}$ and 
 $k\le k_j$. Therefore, we have 
\begin{align*}
  v_k=\Gipolk[k]w\leftarrow \Gipolk[k]\Cipolk[k_j]u_{k_j}=\Gipolk[k_j]\Cipolk[k_j]u_{k_j} =u_{k_j}\to
  \ovu\qquad \text{on}~\elm\in \gridk^{1+}
\end{align*}
as $j\to\infty$ and thus
\begin{align*}
  \sum_{\side\in\sides^+}\int_{\side} \hG[+]^{-1} \jumpn{}{\ovu}^2
  \ds&=\lim_{k\to\infty}\sum_{\side\in\sides^{1+}_k}\int_{\side}
  \hG[k]^{-1} \jumpn{}{\ovu}^2 \ds
  \\
  &=\lim_{k\to\infty}\sum_{\side\in\sides^{1+}_k}\int_{\side}
  \hG[k]^{-1} \jumpn{}{v_k}^2 \ds\le \sup_{k}\enorm[k]{v_k}^2<\infty.
\end{align*}
In the same vein, we have that $\DG^2 v_k|_{\Omega_k^{1+}} = \DG^2
\ovu|_{\Omega_k^{1+}}$,
which implies $\DG^2 v_k \to \DG^2 \ovu$
a.e. in $\Omega^+$ as $k\to\infty$ and thus $\DG^2 \ovu\in
L^2(\Omega^+)$ thanks to Fatou's Lemma.
Together with $D^2 \ovu=D^2 w$ in $\Omega^-$ from \step{1}, this yields $\DG^2
\ovu\in L^2(\Omega)$ and we conclude $\enorm[\infty]{\ovu}<\infty$.

The assertion follows now by
splitting $\enorm[k]{v_k -\ovu}^2$ according to  $
  \gridk= \gridk^{2-} \cup
  \gridk^{2+} $ and investigating the resulting terms separately
similar to step \step{4} in the proof of
  Lemma~\ref{lem:limit_space_is_Hilbert}.
\end{proof}
In order to prove that $\ovu$ solves \eqref{eq:C0IP_limit}, we need to
identify the limit of its distributional derivatives. To this end, we note that
by~\eqref{est:global_liftings} and~\eqref{est:uniform_stab_solution}
we have $\norm[\Omega]{\DG^2 
  u_{k_j}}\Cleq 1$ and $\norm[\Omega]{\liftG[k_j](\jump{\pn u_{k_j}}
  )}\Cleq 1$. Consequently, there exist $\vec{T_r}, \vec{T_s}\in
L^2(\Omega)^{2 \times 2} $ such that  for a not relabelled
subsequence we obtain 
\begin{align}\label{eq:TsTr}
  \DG^2 u_{k_j} \wconv \vec{T_r}\qquad \text{and}\qquad   \liftG[k_j](\jump{\pn u_{k_j}} )\wconv \vec{T_s}
 \end{align}
weakly in $L^2(\Omega)^{2 \times 2}$ as $j \to \infty$.
\begin{lem}\label{lem:weak_convergence_Omega^-}
  Let $\set{u_{k_j}}_{j \in \N_0}$ be the subsequence of discrete solutions
  with weak $H^1_0(\Omega)$ limit $\ovu \in \V_\infty$
  from~\eqref{conv:u_k_to_ovu_wH1}. 
  Then, we have for $\vec{T_s},\vec{T_r}\in L^2(\Omega)^{2 \times 2}$
  from~\eqref{eq:TsTr} that
  \begin{align*}
     (\vec{T_r} -\vec{T_s})|_{\Omega^-}= D^2
     \ovu|_{\Omega^ -}\qquad \text{a.e. in}~\Omega^-.
  \end{align*}
\end{lem}

\begin{proof}
Proposition \ref{prop:total_variation_of_gradient_bounded_by_energy}
and~\eqref{est:uniform_stab_solution} imply that $\{\nabla u_{k_j}\}_{j\in\N_0}$ is
uniformly bounded in $BV(\Omega)^2$. Hence, as
in~\eqref{eq:uk->ovuBV}, we have that $\nabla u_{k_j} \wconv^* \nabla
\ovu$  in $BV(\Omega)^2$, which implies that the Hessian $D^2
u_{k_j}$ converges to $D^2 \ovu$ in the sense of
distributions; 
compare e.g. with \cite[Chapter 3.1]{Ambrosio:2000}. In particular, for $\vec\varphi  \in  C^\infty_0(\Omega)^{2 \times
  2} $, we have
\begin{align}\label{distri-uk->ovu}
  \qquad\int_{\Omega} \divo\vec \varphi \cdot \nabla u_{k_j}\dx \to
\int_{\Omega} \divo\vec \varphi \cdot\nabla\ovu\dx,\quad\text{as}~j\to\infty.
\end{align}
Using the fact that $\ovu\in\V_\infty$, we have that there exists a
sequence $\{v_k\}_{k\in\N_0}$ with $v_k\in\V_k$, $k\in\N_0$, and
$\enorm[k]{\ovu-v_k}\to 0$ as $k\to \infty$. This implies that
\begin{align}\label{eq:distri-ovu}
  \begin{aligned}
    -\int_{\Omega} \divo\vec \varphi \cdot \nabla \ovu\dx
    &=-\lim_{k\to\infty}\int_{\Omega} \divo\vec \varphi \cdot \nabla
    v_k\dx
    \\
    &= \lim_{k\to\infty}\int_{\Omega} \DG^2 v_k \colon \vec \varphi  
    \dx - \int_{\sidesk[k_j]} \jumpn{}{ v_k} \vec\varphi \normal \cdot
    \normal\ds
    \\
    &= \int_{\Omega} \DG^2 \ovu \colon \vec \varphi  \dx
    - 
    \int_{\sides^+} \jumpn{}{ \ovu}\vec\varphi \normal \cdot \normal
    \ds,
  \end{aligned}
\end{align}
where we have used that $\sum_{\side\in\sidesk\setminus\sidesk^+
  }\int_{\side}
  \jumpn{}{
  \ovu}\vec\varphi \normal \cdot \normal \ds\to 0 $ as $k\to\infty$,
thanks to Proposition~\ref{prop:boundedness_of_energy_norm}.

On the other hand, fix $\ell\in\N_0$, and let
$\vec\pi_{k_j}=\vec\pi_{k_j}(\vec\varphi )$ be the $L^2$-projection 
of $\vec\varphi $ onto $\P_0(\grid_{k_j})^{2\times2}$. Then
for the first  term in \eqref{distri-uk->ovu}
\begin{align*}
  -\int_{\Omega} \divo\vec \varphi \cdot \nabla u_{k_j}\dx &=\int_{\Omega} \DG^2 u_{k_j}
  \colon \vec \varphi  \dx - 
  \int_{\sidesk[k_j]}
  \jumpn{}{
  u_{k_j}} \vec\varphi \normal \cdot \normal \ds
\end{align*}
we have, thanks to the definition of the lifting~\eqref{df:local_liftings}, that
\begin{align}\label{eq:distri-uk}
   \begin{aligned}
  &
\int_{\sidesk[k_j]} \jump{\pn
  u_{k_j}} \vec\varphi \normal \cdot \normal\ds
  \\
    &= 
    \int_{\sidesk[\ell]^+}\jump{\pn u_{k_j}} \vec\varphi\normal \cdot \normal 
    \ds+
    \int_{\sidesk[k_j]\setminus\sidesk[\ell]^+}\jump{\pn u_{k_j}}
    \mean{(\vec\varphi-\vec\pi_{k_j} ) \normal \cdot \normal} 
    \ds
    \\
    &\quad+\int_{\Omega_{\ell}^-} \liftG[k_j](\jump{\pn u_{k_j}})
    \colon(\vec \pi_{k_j} -\vec\varphi )\dx+\int_{\Omega_{\ell}^-} \liftG[k_j](\jump{\pn u_{k_j}})
    \colon\vec\varphi \dx
  \end{aligned}
\end{align}
for all $\ell\le k_j$. Thanks to
Lemma~\ref{lem:Omega_star_and_h_vanishing}, for $\epsilon>0$, we have 
\[\norm[L^\infty(\Omega_\ell^-)]{\vec\varphi-\vec\pi_{k_j}}\leq
\norm[L^\infty(\Omega)]{h_{k_j}\chi_{\Omega_\ell^-}}\norm[L^\infty(\Omega)]{\nabla\vec\varphi}\le
\norm[L^\infty(\Omega)]{h_{\ell}\chi_{\Omega_\ell^-}}\norm[L^\infty(\Omega)]{\nabla\vec\varphi}<\epsilon\]
for sufficiently large $\ell=\ell(\epsilon,\vec\varphi)\le k_j$ and thus
\begin{multline*}
    \abs{\int_{\sidesk[k_j]\setminus\sidesk[\ell]^+}\jump{\pn u_{k_j}} \mean{(\vec\varphi-\vec\pi_{k_j})\normal \cdot \normal }
    \ds
    +\int_{\Omega_{\ell}^-} \liftG[k_j](\jump{\pn u_{k_j}})
    \colon(\vec \pi_{k_j} -\vec\varphi )\dx}
  \\\Cleq \epsilon \norm{f} \norm[L^\infty(\Omega)]{\nabla\vec\phi}.
\end{multline*}
As a consequence of~\eqref{eq:uk->ovuBV} and the fact, that
$u_{k_j}|_{\Omega_\ell^+}\in \P_2(\grid_\ell^+)$ is finite
dimensional, we have that
\begin{multline*}
    \int_{\sidesk[\ell]^+}\jump{\pn u_{k_j}} \vec\varphi\normal \cdot \normal 
    \ds+\int_{\Omega_{\ell}^-} \liftG[k_j](\jump{\pn u_{k_j}})
    \colon\vec\varphi \dx\\
    \to \int_{\sidesk[\ell]^+}\jump{\pn \ovu} \vec\varphi\normal \cdot \normal 
    \ds+\int_{\Omega_{\ell}^-} \vec{ T_s}
    \colon\vec\varphi \dx
\end{multline*}
as $j\to\infty$. Upon choosing $\ell$ even larger, we have also
\begin{align*}
  \abs{\int_{\sides^+\setminus\sidesk[\ell]^+}\jump{\pn \ovu} \vec\varphi\normal \cdot \normal 
    \ds+\int_{\Omega_{\ell}^-\setminus\Omega^-} \vec{ T_s}
    \colon\vec\varphi \dx}<\epsilon.
\end{align*}
Inserting this in~\eqref{eq:distri-uk}, we have thanks to the fact
that $\epsilon>0$ was arbitrary, that
\begin{align*}
  -\int_{\Omega} \divo\vec \varphi \cdot \nabla u_{k_j}\dx &\to
                                                             \int_\Omega
                                                             \vec{T_r}
                                                             \colon\vec\varphi
                                                              \dx-\int_{\Omega^-}\vec{T_s}
                                                             \colon\vec\varphi
                                                              \dx-\int_{\sides^+}\jump{\pn\ovu}\colon
                                                             \vec\varphi\normal\cdot\normal\ds. 
\end{align*}
In view of~\eqref{distri-uk->ovu} and~\eqref{eq:distri-ovu}, this thus
implies that
\begin{align*}
 0&= \lim_{j\to\infty}\int_{\Omega} \divo\vec \varphi \cdot \nabla
    (u_{k_j}-\ovu)\dx =\int_\Omega (\DG^2\ovu-\vec{T_r}+\vec
    {T_s}\chi_{\Omega^-})\colon\vec\varphi \dx
\end{align*}
for all $\vec\varphi\in C_0^\infty(\Omega)^{2\times 2}$. The desired
assertion follows from the density of $C_0^\infty(\Omega)$
in $L^2(\Omega)$.
\end{proof}

Now, we are in the position to conclude that $\ovu$ and $u_\infty$ coincide.

\begin{lem}\label{lem:ovu=u_infty}
  We have that $\ovu \in \V_\infty$ 
  solves~\eqref{eq:C0IP_limit} and
  thus $\ovu=u_\infty$. In particular, the
  limit in~\eqref{conv:u_k_to_ovu_wH1} is unique and the full sequence
  $\{u_k\}_{k\in\N_0}$ converges to $u_\infty$ weakly in $H_0^1(\Omega)$.
\end{lem}

\begin{proof}
  Let  $v \in \V_\infty$ and 
  $\set{v_k}_{k \in\N_0}$, $v_k \in \V_k$ such that $\enorm[k]{v_k-v}\to 0$ as $k \to
  \infty$. Consequently, for the subsequence~\eqref{conv:u_k_to_ovu_wH1} of discrete
  solutions   $\set{u_{k_j}}_{j \in \N_0}$, we
have  
\begin{align}\label{eq:conv_right_hand_side_bilf}
\bilin[k_j]{u_{k_j}}{v_{k_j}}=\scp[\Omega]{f}{v_{k_j}}\to
  \scp[L^2(\Omega)]{f}{v} \quad \text{as }j \to \infty.
\end{align}
Using $\enorm[k]{v_k-v}\to 0$ as $k \to \infty$ again, it suffices to
prove $  \bilin[k_j]{u_{k_j}}{v} 
\to\bilin[\infty]{\ovu}{v}$ as $j \to \infty$. 

To see this, we 
split the bilinear form according to
  \begin{align*}
    \begin{aligned}
    \bilin[k_j]{u_{k_j}}{v}
&= \int_\Omega (\DG^2 u_{k_j}-\liftG[k_j](\jump{\pn u_{k_j}}) \colon \DG^2v \dx
    \\
&\quad-\int_\Omega \liftG[k_j](\jump{\pn v}) \colon \DG^2 u_{k_j} \dx+
\int_{\sidesk[k_j]} \frac{\sigma}{\hG[k_j]}\jumpn{}{u_{k_j}}\jumpn{}{v }\ds \\
&\asdefined I_j- II_j+III_j.
    \end{aligned}
  \end{align*}
and consider the limit of each term separately.

\boxed{1}~From Lemma~\ref{lem:weak_convergence_Omega^-}, we have 
\begin{align}\label{dist->Omega-}
  \int_{\Omega^-} \big(\DG^2 u_{k_j}-\liftG[k_j](\jump{\pn u_{k_j}})\big)
  \colon \DG^2v \dx \to \int_{\Omega^-} D^2 \ovu \colon \DG^2v \dx
\end{align}

For $\ell\le k_j$ we split the domain $\Omega$ according to  
\begin{align*}
  \overline{\Omega}=\overline{\Omega^-} \cup\overline{\Omega_\ell^{1-}\setminus\Omega^-}\cup \overline{\Omega_\ell^{1+}}.
\end{align*}

On $\Omega_\ell^-\setminus\Omega^-$, by uniform integrability of
$\DG^2v$,  Lemma~\ref{lem:Omega_star_and_h_vanishing} and the 
stability of liftings~\eqref{est:global_liftings}, for
$\epsilon>0$ there exists $K(\epsilon)$ such that for all $\ell\ge
K(\epsilon)$, we have 
\begin{multline*}
  \Big|\int_{ \Omega_\ell^{1-} \setminus\Omega^-
 } \Big(\DG^2
  u_{k_j}-\liftG[k_j](\jump{\pn u_{k_j}})-\DG^2
  \ovu+\liftG[\infty](\jump{\pn \ovu}) \Big) \colon \DG^2v \dx \Big|\\
  \begin{aligned}
    & \Cleq \left(
      \enorm[k_j]{u_{k_j}}
      +\enorm[\infty]{\ovu}\right) \norm[\Omega_\ell^{1-} \setminus\Omega^-]{\DG^2v} \le \epsilon.
  \end{aligned}
\end{multline*}
From~\eqref{conv:u_k_to_ovu_wH1} we observe on $\Omega_\ell^+$ that 
$\DG^2u_{k_j}|_{\Omega_\ell^{1+}} \to \DG^2\ovu|_{\Omega_\ell^{1+}}$
strongly in $L^2(\Omega_\ell^{1+})$ as
$j\to\infty$ since
$\polspace[0](\gridk[\ell]^{1+})^{2\times2}$ is finite dimensional for fixed
$\ell$. Therefore, we have
\begin{align*}
  \int_{\Omega_\ell^{1+}} \DG^2 u_{k_j}\colon
  \DG^2v \dx \to  \int_{\Omega_\ell^{1+}} \DG^2 \ovu  \colon
  \DG^2v \dx \quad \text{as }j \to \infty.
\end{align*}
Similar arguments prove $\jump{\pn u_{k_j}}|_{\sides_\ell^{1+}} \to
\jump{\pn \ovu}|_{\sides_\ell^{1+}}$ strongly in
$L^2(\sides_\ell^{1+})$ as $j\to\infty$  and, thanks to the fact
that the local definition~\eqref{df:local_liftings} of the liftings eventually does
not change on $\grid_\ell^{1+}$, we have 
\begin{align*}
  \int_{\Omega_\ell^{1+}}\liftG[k_j](\jump{\pn u_{k_j}}) \colon \DG^2v
  \dx&=\int_{\Omega_\ell^{1+}}\liftG[\infty](\jump{\pn u_{k_j}}) \colon
  \DG^2v \dx
  \\
  &\to
  \int_{\Omega_\ell^{1+}}\liftG[\infty](\jump{\pn \ovu}) \colon \DG^2v \qquad\text{as $j\to\infty$.}
\end{align*}
From the estimate 
\begin{multline*}
   \Big|\int_\Omega \bigg(\DG^2 u_{k_j}-\liftG[k_j](\jump{\pn
     u_{k_j}})  - \DG^2 \ovu  + \liftG[\infty](\jump{\pn \ovu}) \bigg)
  \colon \DG^2v \dx\Big|\\
  \begin{aligned}
  & \leq    
\Big|\int_{\Omega^-} \bigg(\DG^2 u_{k_j}-\liftG[k_j](\jump{\pn
     u_{k_j}})  - \DG^2 \ovu \bigg)
  \colon \DG^2v \dx\Big| \\
& \quad +  
\Big|\int_{\Omega_\ell^{1- }  \setminus \Omega^-}  \bigg(\DG^2 u_{k_j}-\liftG[k_j](\jump{\pn
     u_{k_j}})  - \DG^2 \ovu  + \liftG[\infty](\jump{\pn \ovu}) \bigg)
  \colon \DG^2v \dx\Big| \\
& \quad +  
\Big|\int_{\Omega_\ell^{1+ }}  \bigg(\DG^2 u_{k_j}-\liftG[k_j](\jump{\pn
     u_{k_j}})  - \DG^2 \ovu  + \liftG[\infty](\jump{\pn \ovu}) \bigg)
  \colon \DG^2v \dx\Big|,
  \end{aligned}
\end{multline*}
we finally observe
that the first and third terms vanish as $ j \to \infty$, and arrive at 
\begin{align*}
   \lim_{j \to \infty} \Big|\int_\Omega \bigg(\DG^2 u_{k_j}-\liftG[k_j](\jump{\pn
     u_{k_j}})  - \DG^2 \ovu  + \liftG[\infty](\jump{\pn \ovu}) \bigg)
  \colon \DG^2v \dx\Big|< \epsilon.
\end{align*}
Since $\epsilon>0$ was chosen arbitrarily, for $j \to \infty$, we conclude 
\begin{align}\label{eq:3}
  \int_\Omega (\DG^2 u_{k_j}-\liftG[k_j](\jump{\pn
     u_{k_j}})\colon \DG^2 v \dx   \to \int_\Omega  (\DG^2 \ovu  + \liftG[\infty](\jump{\pn \ovu}) )
  \colon \DG^2v \dx
\end{align}

\boxed{2}~In order to identify the limit of $II_j$, we split the
domain
$\Omega$ according to
 \begin{align*}
   \Omega=(\Omega \setminus
 \Omega_\ell^{1+} )
\cup \Omega_\ell^{1+}
 \end{align*}
 for some $\ell\le k_j$.
Thanks to uniform boundedness 
$\enorm[k]{u_k}\Cleq \norm[\Omega]{f}$, for $\epsilon>0$, we have 
\begin{align}\label{eq:lift_k_v_D2uk}
  \Big|\int_{\Omega\setminus \Omega_\ell^{1+}} \liftG[k_j](\jump{\pn v}) \colon \DG^2 u_{k_j}
  \dx \Big| \Cleq \norm[\Omega \setminus
  \Omega_\ell^{1+}]{\liftG[k_j](\jump{\pn v})}  \norm[\Omega]{f}<\epsilon
\end{align}
for all $k_j\ge\ell\ge K(\epsilon)$. Indeed,  the stability of the lifting
operator~\eqref{est:local_liftings} together with Proposition
\ref{prop:boundedness_of_energy_norm} 
yields
\begin{align*}
\norm[\Omega
                                                 \setminus
                                                 \Omega_\ell^{1+}]{\liftG[k_j](\jump{\pn
                                                 v})}
 \Cleq \left( \int_{\sidesk[k_j]\setminus \sidesk[\ell]^{2+}} \hG[k_j]^{-1}
  \jumpn{}{v}^2   \ds\right)^{1/2}
 \to 0\qquad\text{as}~k_j\ge\ell\to\infty.
\end{align*}
As in \step{1}, on $\Omega_\ell^{1+} $ we employ the strong convergence 
$\DG^2u_{k_j}|_{\Omega_\ell^{1+}}\to \DG^2\ovu|_{\Omega_\ell^{1+}}\in
\polspace[0](\gridk[\ell]^{1+})^{2\times 2}$ in
$L^2(\Omega_\ell^{1+})$ as
$j\to\infty$, in order to obtain from the local definitions of the liftings
\eqref{df:local_liftings} and 
\eqref{eq:df_local_lifting_infty} that
\begin{align*}
  \int_{\Omega_\ell^{1+}} 
\liftG[k_j](\jump{\pn v}) \colon \DG^2 u_{k_j} \dx&=   \int_{\Omega_\ell^{1+}} 
                                              \liftG[\infty](\jump{\pn v}) \colon \DG^2 u_{k_j} \dx
  \\
                                            &\to \int_{\Omega_\ell^{1+}} 
\liftG[\infty](\jump{\pn v}) \colon \DG^2 \ovu \dx \qquad\text{as}~j\to\infty.
\end{align*}
Combining this with~\eqref{eq:lift_k_v_D2uk}~
yields
\begin{align}
\int_\Omega \liftG[k_j](\jump{\pn v}) \colon \DG^2 u_{k_j}\dx\to
\int_\Omega \liftG[\infty](\jump{\pn v}) \colon \DG^2
\ovu\dx\quad\text{as $k\to\infty$.}\label{eq:4}
\end{align}

\boxed{3}~For the last term $III_j$, we observe from
$\sides_{\ell}^+\subset\sidesk[k_j]^+$, $\ell\le k_j$, that 
\begin{align*}
    \int_{\sidesk[k_j]} \frac{\sigma}{\hG[k_j]}\jumpn{}{u_{k_j}}\jumpn{}{v
    }\ds
  &= \int_{\sidesk[\ell]^+ }
  \frac{\sigma}{\hG[k_j]}\jumpn{}{u_{k_j}}\jumpn{}{v }\ds
  \\
  &\quad+ \int_{ \sidesk[k_j]
          \setminus \sidesk[\ell]^+ }
\frac{\sigma}{\hG[k_j]}\jumpn{}{u_{k_j}}\jumpn{}{v }\ds.
\end{align*}
For the second term on the right-hand side, we conclude from
Proposition~\ref{prop:boundedness_of_energy_norm} that for arbitrary fixed
$\epsilon$ there exists $K(\epsilon)>0$  such that 
\begin{multline*}
\int_{ \sidesk[k_j]
          \setminus \sidesk[\ell]^+ }
  \frac{\sigma}{\hG[k_j]}\jumpn{}{u_{k_j}}\jumpn{}{v }\ds\\
  \begin{aligned}
    &\leq \left( \int_{ \sidesk[k_j] \setminus \sidesk[\ell]^+ }
      \frac{\sigma}{\hG[k_j]}\jumpn{}{u_{k_j}}^2\ds\right)^{1/2}\left(
      \int_{ \sidesk[k_j] \setminus \sidesk[\ell]^+ }
      \frac{\sigma}{\hG[k_j]}\jumpn{}{v}^2\ds\right)^{1/2} \\
    &\Cleq \enorm[k_j]{u_{k_j}}\left( \int_{ \sidesk[k_j] \setminus
        \sidesk[\ell]^+ }
      \frac{\sigma}{\hG[k_j]}\jumpn{}{v}^2\ds\right)^{1/2}\\
    &\Cleq
    \norm[L^2(\Omega)]{f} \left( \int_{ \sides^+ \setminus
        \sidesk[\ell]^+ }
      \frac{\sigma}{\hG[+]}\jumpn{}{v}^2\ds\right)^{1/2}
    \le \epsilon
  \end{aligned}
\end{multline*}
whenever $k_j\ge\ell\ge K(\epsilon)$. As in \step{1}, we use for fixed
$\ell$  that
$\jump{\pn u_{k_j}}|_{\sides_\ell^{1+}} \to 
\jump{\pn \ovu}|_{\sides_\ell^{1+}}$ as $j\to\infty$ strongly in $L^2(\sides_\ell^{1+})$
 and consequently
\begin{align*}
   \int_{\sidesk[\ell]^+ }
\frac{\sigma}{\hG[k_j]}\jumpn{}{u_{k_j}}\jumpn{}{v }\ds \to    \int_{\sidesk[\ell]^+ }
\frac{\sigma}{\hG[+]}\jumpn{}{\ovu}\jumpn{}{v }\ds \qquad\text{as $j\to\infty$}.
\end{align*}
Since $\epsilon>0$ was arbitrary, the desired convergence
\begin{align}\label{eq:5}
  \int_{\sidesk[k_j]} \frac{\sigma}{\hG[k_j]}\jumpn{}{u_{k_j}}\jumpn{}{v
    }\ds\to \int_{\sidesk[+]} \frac{\sigma}{\hG[+]}\jumpn{}{\ovu}\jumpn{}{v
    }\ds\qquad\text{as}~j\to\infty
\end{align}
follows from $\int_{\sidesk[+]\setminus\sidesk[\ell]^+} \frac{\sigma}{\hG[+]}\jumpn{}{\ovu}\jumpn{}{v
    }\ds \to 0$ as $\ell\to\infty$.

\boxed{4}
Finally, combining~\eqref{eq:3},~\eqref{eq:4} and~\eqref{eq:5}, we have proved 
\begin{align*}
  \bilin[k_j]{u_{k_j}}{v}
&\to \int_{\Omega^-} D^2 \ovu \colon D^2v \dx +   \int_{\Omega^+}(
  \DG^2 \ovu -\liftG[\infty](\jump{\pn \ovu})\colon \DG^2v \dx \\
&\quad + \int_{\Omega^+} \liftG[\infty](\jump{\pn v}) \colon \DG^2 \ovu
  \dx +\int_{\sides^+ }
\frac{\sigma}{\hG[+]}\jumpn{}{\ovu}\jumpn{}{v }\ds \\
&=\bilin[\infty]{\ovu}{v} \qquad \text{as }j \to \infty.
\end{align*}
Hence, by~\eqref{eq:conv_right_hand_side_bilf} we have
$\ovu=u_\infty$, thanks to  $\ovu \in \V_\infty$ and the
uniqueness of the generalised Galerkin solution~\eqref{eq:C0IP_limit}.
\end{proof}

We conclude the section by finally proving Theorem \ref{thm:u_k_to_u_infty}.

\begin{proof}[Proof of Theorem \ref{thm:u_k_to_u_infty}]
  Using the coercivity of the bilinear form,
  Corollary~\ref{cor:convergence_interpolation} and Lemma~\ref{lem:ovu=u_infty},
    and the interpolation
  operator $\ipolk u_\infty \in \V_k$, we
  observe 
  \begin{align*}
      \begin{aligned}
    C_{coer} \enorm[k]{\ipolk u_\infty -u_k}^2 &\leq  \bilin[k]{\ipolk
                                                    u_\infty
                                                    -u_k}{\ipolk
                                                    u_\infty -u_k}\\
&= \bilin[k]{\ipolk u_\infty}{\ipolk
   u_\infty}-2\bilin[k]{\ipolk u_\infty}{u_k}+\bilin[k]{u_k}{u_k}\\
    &=  \bilin[k]{\ipolk u_\infty}{\ipolk
   u_\infty}-2\scp[L^2(\Omega)]{f}{\ipolk
       u_\infty}+\scp[L^2(\Omega)]{f}{u_k}\\
    &\to
    \bilin[\infty]{u_\infty}{u_\infty}-\scp[L^2(\Omega)]{f}{u_\infty}=0 \qquad
  \text{as } k \to \infty.
      \end{aligned}
  \end{align*}
Hence, again with  Corollary~\ref{cor:convergence_interpolation}, we conclude
\begin{align*}
  \enorm[k]{ u_\infty -u_k}^2\leq \enorm[k]{\ipolk
  u_\infty -u_\infty}^2+\enorm[k]{\ipolk u_\infty -u_k}^2 \to
  0 
\end{align*} 
as $k \to \infty$.
\end{proof}

\section*{Acknowledgments}
Alexander Dominicus and Christian Kreuzer gratefully acknowledge
partial support by the DFG research grant KR 3984/5-1 ``Convergence
Analysis for Adaptive Discontinuous Galerkin Methods''.

We also thank the anonymous referees for finding a highly non-trivial
counterexample to a statement in a previous version of this article.


\newcommand{\etalchar}[1]{$^{#1}$}
\providecommand{\bysame}{\leavevmode\hbox to3em{\hrulefill}\thinspace}
\providecommand{\MR}{\relax\ifhmode\unskip\space\fi MR }
\providecommand{\MRhref}[2]{%
  \href{http://www.ams.org/mathscinet-getitem?mr=#1}{#2}
}
\providecommand{\href}[2]{#2}


\appendix
\printnomenclature[5em]
\end{document}